\documentclass[12pt, reqno, parskip]{amsart}
\usepackage{amssymb, amsthm, amsmath, amsfonts, graphicx}
\usepackage{array, epsfig}
\usepackage{bbm}
\usepackage{hyperref}
\usepackage[numbers,square]{natbib}
\usepackage{color}
\usepackage{parskip}

\usepackage{manfnt} %contains cube
\usepackage{tikz} %for Tetrahedron and Octahedron

  \evensidemargin
\oddsidemargin
\newcommand{\codim}{\mathop{\mathrm{codim}}\nolimits}
\newcommand{\dist}{\mathop{\mathrm{dist}}\nolimits}

\newcommand{\cF}{\mathcal{F}}

\newcommand{\Normal}{\mathcal{N}}

\newcommand{\Gr}{\mathrm{Gr}}

\DeclareMathOperator{\relint}{relint}
\DeclareMathOperator{\interior}{int}

\newcommand{\E}{\mathbb E}

\newcommand{\R}{\mathbb{R}}
\newcommand{\N}{\mathbb{N}}

\renewcommand{\P}{\mathbb{P}}

\newcommand{\Span}{\mathop{\mathrm{lin}}\nolimits}

\newcommand{\Cov}{\mathop{\mathrm{Cov}}\nolimits}

\newcommand{\sgn}{\mathop{\mathrm{sgn}}\nolimits}
\newcommand{\conv}{\mathop{\mathrm{conv}}\nolimits}
\newcommand{\pos}{\mathop{\mathrm{pos}}\nolimits}
\newcommand{\aff}{\mathop{\mathrm{aff}}\nolimits}
\newcommand{\lin}{\mathop{\mathrm{lin}}\nolimits}

\newcommand{\eps}{\varepsilon}

\newcommand{\ton}{\overset{}{\underset{n\to\infty}\longrightarrow}}

\newcommand{\dd}{{\rm d}}
\newcommand{\eee}{{\rm e}}

\newcommand{\indi}[1]{1_{ \left\{ #1 \right\} }}

\theoremstyle{plain}
\newtheorem{theorem}{Theorem}[section]
\newtheorem{lemma}[theorem]{Lemma}
\newtheorem{corollary}[theorem]{Corollary}
\newtheorem{proposition}[theorem]{Proposition}

\theoremstyle{definition}

\theoremstyle{remark}

\newcommand{\Cs}{C^{\tetrahedron}}
\newcommand{\Ccr}{C^{\octahedron}}
\newcommand{\Ccu}{C^{\mbox{\mancube}}}

\newcommand{\gs}{g^{\tetrahedron}}
\newcommand{\gcr}{g^{\octahedron}}
\newcommand{\gcu}{g^{\mbox{\mancube}}}

\newcommand{\Lineal}{\operatorname{Lineal}}

\newcommand{\Pgau}{P^{\tetrahedron}}
\newcommand{\fgau}{f^{\tetrahedron}}

\newcommand{\Psym}{P^{\octahedron}}
\newcommand{\fsym}{f^{\octahedron}}
\newcommand{\psym}{p^{\octahedron}}
\newcommand{\bcr}{b^{\octahedron}}
\newcommand{\bsym}{\bcr}

\newcommand{\Pzono}{P^{\mbox{\mancube}}}
\newcommand{\fzono}{f^{\mbox{\mancube}}}
\newcommand{\pzono}{p^{\mbox{\mancube}}}
\newcommand{\bcu}{b^{\mbox{\mancube}}}
\newcommand{\bzono}{\bcu}

\newcommand{\phicu}{\phi^{\mbox{\mancube}}}
\newcommand{\phis}{\phi^{\tetrahedron}}
\newcommand{\phicr}{\phi^{\octahedron}}
\newcommand{\phip}{\phi^P}

\newcommand{\gKZ}{g}

\def\factor{2/3*0.1}
\newcommand{\octahedron}{
	\begin{tikzpicture}[thick,scale=5]
	\coordinate (A1) at (0,0);
	\coordinate (A2) at (0.6*\factor,0.2*\factor);
	\coordinate (A3) at (1*\factor,0);
	\coordinate (A4) at (0.4*\factor,-0.2*\factor);
	\coordinate (B1) at (0.5*\factor,0.5*\factor);
	\coordinate (B2) at (0.5*\factor,-0.5*\factor);
	
	\draw[solid][line width=0.35pt] (A1) -- (A4) -- (B1);
	\draw[solid][line width=0.35pt] (A1) -- (A4) -- (B2);
	\draw[solid][line width=0.35pt] (A3) -- (A4) -- (B1);
	\draw[solid][line width=0.35pt] (A3) -- (A4) -- (B2);
	\draw[solid][line width=0.35pt] (B1) -- (A1) -- (B2) -- (A3) --cycle;
	\end{tikzpicture}
}

\def\factortetra{0.1}
\newcommand{\tetrahedron}{
\begin{tikzpicture}[thick,scale=5]

\coordinate (C1) at (0.2*\factortetra,0);
\coordinate (C2) at (0.8*\factortetra,-0.1*\factortetra);

\coordinate (A5) at (0.4*\factortetra,-0.2*\factortetra);
\coordinate (B3) at (0.5*\factortetra,0.4*\factortetra);

\draw[solid][line width=0.5pt] (C1) -- (A5) -- (B3);
\draw[solid][line width=0.5pt] (C2) -- (A5); %-- (B1);
\draw[solid][line width=0.5pt] (C1) -- (B3) -- (C2);

\end{tikzpicture}
}

\begin{document}
%\hyphenation{U-ni-ver-si-t\"at}

\author{Zakhar Kabluchko}
\address{Zakhar Kabluchko: Institut f\"ur Mathematische Stochastik,
Westf\"alische {Wilhelms-Uni\-ver\-sit\"at} M\"unster,
Orl\'eans--Ring 10,
48149 M\"unster, Germany}
\email{zakhar.kabluchko@uni-muenster.de}

\author{Hauke Seidel}
\address{Hauke Seidel: Institut f\"ur Mathematische Stochastik,
Westf\"alische Wilhelms-Universit\"at M\"unster,
Orl\'eans--Ring 10,
48149 M\"unster, Germany}
\email{hauke.seidel@uni-muenster.de}

\title[Convex cones spanned by regular polytopes]{Convex cones spanned by regular polytopes}

\keywords{Convex cone, regular polytope, solid angle, conic intrinsic volume, Gaussian polytope, absorption probability, sections of regular polytopes, projections of regular polytopes}

%\thanks{}

\subjclass[2010]{Primary: 60D05; secondary: 52A22,	52A55,	52B11}

\begin{abstract}
We study three families of polyhedral cones whose sections are regular simplices, cubes, and crosspolytopes. We compute  solid angles and conic intrinsic volumes of these cones. We show that several quantities appearing in stochastic geometry can be expressed through these conic intrinsic volumes. A list of such quantities includes internal and external solid angles of regular simplices and crosspolytopes, the probability that a (symmetric) Gaussian random polytope or the Gaussian zonotope contains a given point, the expected number of faces of the intersection of a regular polytope with a random linear subspace passing through its centre, and the expected number of faces of the projection of a regular polytope onto a random linear subspace.
\end{abstract}

\maketitle
%$C^{
%\begin{tikzpicture}[thick,scale=5]
%\coordinate (A1) at (0,0);
%\coordinate (A2) at (0.6*\factor,0.2*\factor);
%\coordinate (A3) at (1*\factor,0);
%\coordinate (A4) at (0.4*\factor,-0.2*\factor);
%\coordinate (B1) at (0.5*\factor,0.5*\factor);
%\coordinate (B2) at (0.5*\factor,-0.5*\factor);
%
%%\begin{scope}[thin,dashed,opacity=0.06]
%%\draw (A1) -- (A2) -- (A3);
%%\draw (B1) -- (A2) -- (B2);
%%\end{scope}
%
%\draw[solid][line width=0.2pt] (A1) -- (A4) -- (B1);
%\draw[solid][line width=0.2pt] (A1) -- (A4) -- (B2);
%\draw[solid][line width=0.2pt] (A3) -- (A4) -- (B1);
%\draw[solid][line width=0.2pt] (A3) -- (A4) -- (B2);
%\draw[solid][line width=0.2pt] (B1) -- (A1) -- (B2) -- (A3) --cycle;
%\end{tikzpicture}
%}+
%$\Ccu+C^{\octahedron}+\tetrahedron + C^{\tetrahedron}$
%+7\cdot \Ccu$
\section{Definition of the cones}\label{sec:defnition_cones}
\subsection{Introduction}
In Euclidean geometry, there are three infinite series of regular polytopes: regular simplices, regular crosspolytopes and  cubes. In this paper, we shall be interested in convex cones whose ``sections''  are these regular polytopes.
%These cones (and their solid angles) appear in various problems of stochastic geometry such as
It turns out that many quantities appearing in stochastic geometry can be related to the solid angles of these cones. These quantities include
\begin{enumerate}
\item Internal and external angles of the regular simplex and the regular crosspolytope.
\item Absorption probabilities for certain Gaussian random polytopes.
\item Expected number of faces of a regular polytope intersected by a random linear subspace.
\item Expected number of faces of a random projection of a regular polytope.
\end{enumerate}
The paper is organized as follows. In the remaining part of the present Section~\ref{sec:defnition_cones} we introduce the cones we are interested in. In Section~\ref{sec:cones_angles} we compute the solid angles and the conic intrinsic volumes of these cones. In Section~\ref{sec:absorption_probab} we relate absorption probabilities of Gaussian random polytopes to the cones we are interested in. In Section~\ref{sec:random_sections} we express the number of faces in an intersection of a regular polytope by a random linear subspace through the conic intrinsic volumes of our cones. Finally, the proofs  are collected in  Section~\ref{sec:proofs}.

\subsection{Notation}
The $n$-dimensional Euclidean space is denoted by $\R^n$ and equipped with the standard scalar product $\langle \cdot, \cdot\rangle$.  The standard orthonormal basis of $\R^n$ is denoted by $\eee_1,\ldots, \eee_{n}$. For a non-empty set $A\subset \R^n$ its \textit{convex hull} and its \textit{positive hull} are defined by
\begin{align*}
\conv A &= \{\lambda_1 a_1+\ldots+ \lambda_k a_k \colon k\in \N, a_1,\ldots,a_k \in A, \lambda_1,\ldots,\lambda_k \geq 0, \lambda_1+\ldots+\lambda_k =1\},\\
\pos A  &= \{\lambda_1 a_1+\ldots+ \lambda_k a_k \colon k\in \N, a_1,\ldots,a_k \in A, \lambda_1,\ldots,\lambda_k \geq 0\}.
\end{align*}
The linear space spanned by $A$ is denoted by $\lin A$. For a subset $F\subset \R^n$, we denote by $\relint (F)$ its relative interior, that is the interior of $F$ with respect to its affine hull $\aff(F)$. A \textit{polytope} is a convex hull of finitely many points in $\R^n$. Similarly, a   \textit{polyhedral cone} (or just a \textit{cone}, for the purposes of the present paper) is a positive hull of finitely many points in $\R^n$. Alternatively, a polytope can be defined as an intersection of finitely many closed halfspaces (provided it is bounded), whereas a cone is an intersection of finitely many closed half-spaces whose bounding hyperplanes pass through the origin. For general references on convex sets, polytopes, and stochastic geometry we refer to the monographs~\cite{schneider_convex_bodies}, \cite{ziegler_polytopes}, \cite{SW08}.
%is an intersection of finitely many closed half-spaces whose boundaries pass through the origin. That is, a cone is the set of the form
%$$
%\{x\in\R^n \colon \langle x,v_1\rangle\leq 0, \ldots , \langle x,v_k\rangle\leq 0\},
%$$
%where $v_1,\ldots,v_k\in\R^n$. Alternatively, a cone

\subsection{Cones associated with regular polytopes}
There are three kinds of cones we are interested in. Their definitions are all motivated in the following way. Let $P\subset \R^n$ be a regular polytope. Identify the space $\R^n$ with a hyperplane in $\R^{n+1}$ spanned by the standard orthonormal basis vectors $\eee_1,\ldots,\eee_n$ and shift the polytope by some distance $\sigma>0$ in direction of the last $(n+1)$\textsuperscript{th} basis vector $\eee_{n+1}$. Our cones are the smallest cones that contain the corresponding polytope and have their apex at the origin. We are interested in the three cases when $P$ is a simplex $P_n=\conv\{\eee_1,\ldots, \eee_n\}$, a crosspolytope $P_n=\conv\{\pm \eee_1,\ldots, \pm \eee_n\}$ or a cube $P_n = [-1,1]^n$.
The corresponding cones can be formally defined in the following way:
\begin{align}
\Cs_n(\sigma^2):&=\pos (\sigma \eee_{n+1}+\eee_j:
j\in\{1,\ldots,n\})\label{eq:def_Cs_n},
\\
\Ccr_n(\sigma^2)
%:&=\pos (\sigma \eee_{n+1}\pm \eee_j:j\in\{1,\ldots,n\})  \\
:&=\pos (\sigma \eee_{n+1}+ \eee_j, \sigma \eee_{n+1} - \eee_j : j\in\{1,\ldots,n\}), \label{eq:def_Ccr_n} %\notag
\\
\Ccu_n(\sigma^2)
:&=\pos \Big(\sigma \eee_{n+1}+\sum_{j=1}^n \eps_j \eee_j:
\eps= (\eps_1,\ldots,\eps_n)\in\{-1,1\}^n\Big); \label{eq:def_Ccu_n}
\end{align}
see Figure~\ref{fig:cones}.
\begin{figure}
	\includegraphics[height=0.33\textheight]{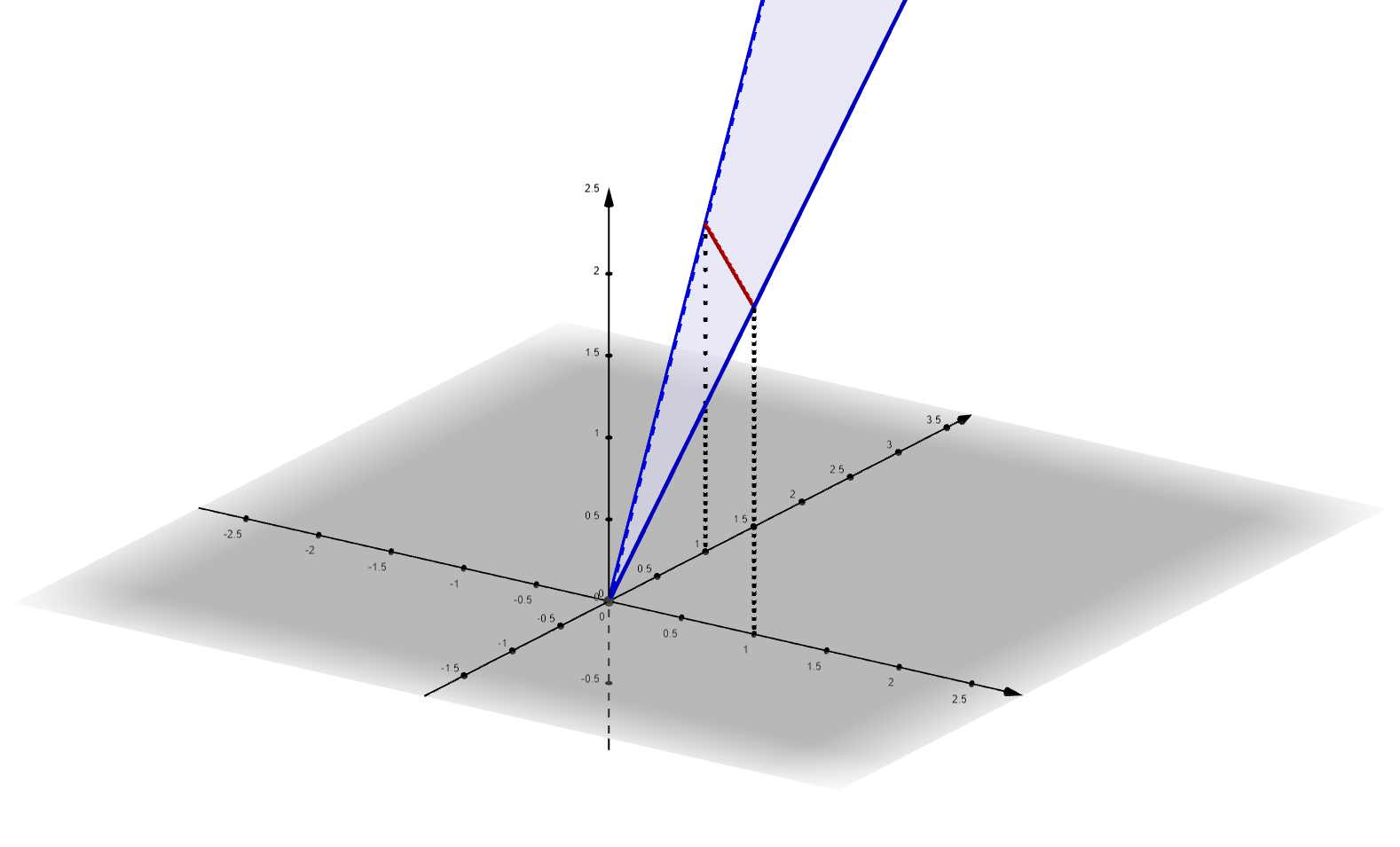}
	\includegraphics[height=0.33\textheight]{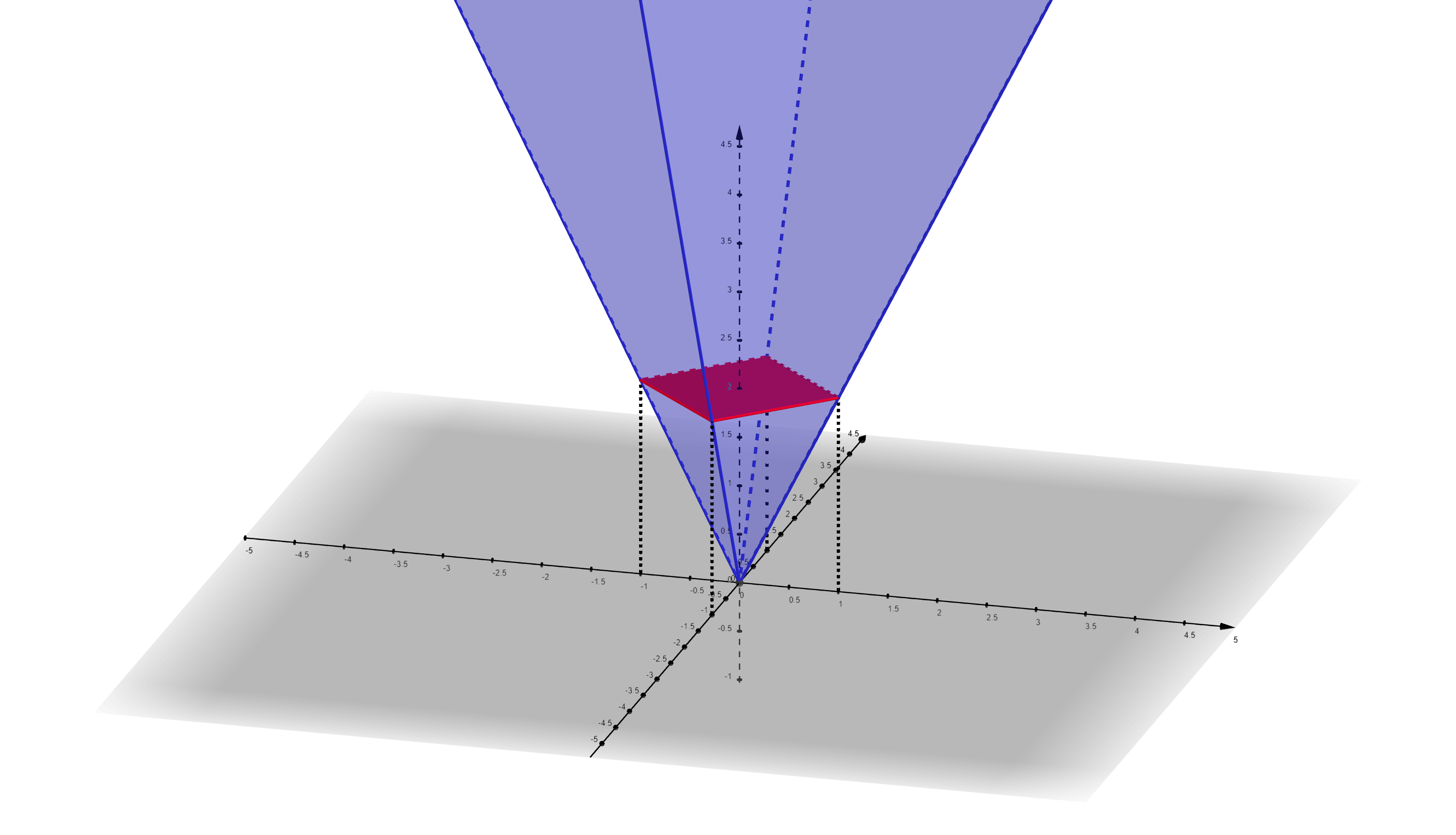}
	\includegraphics[height=0.33\textheight]{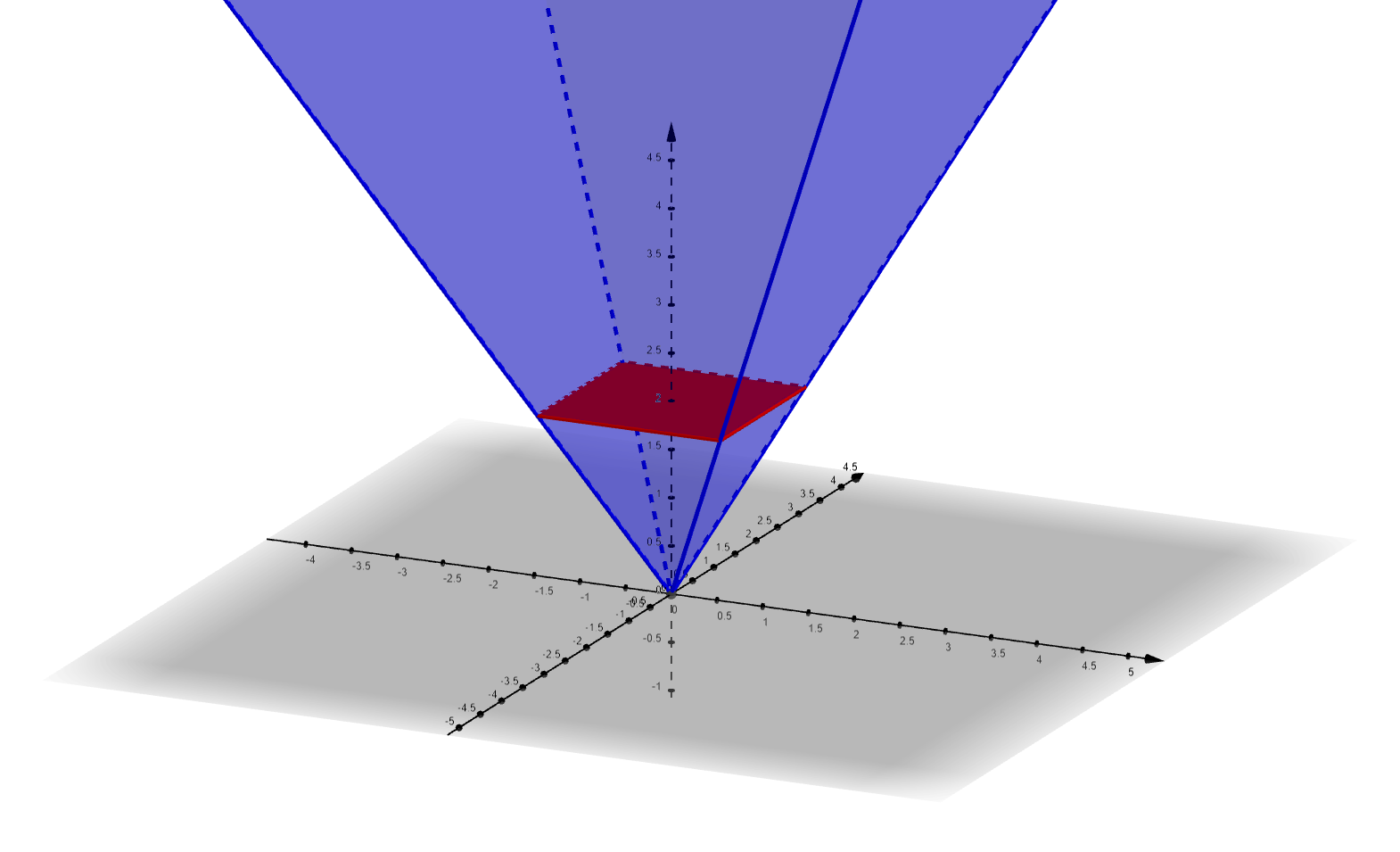}
%	%\caption{The cones with index 2 and parameter 4 of the types ``simplex'', ``cube'' and ``cross-polytope''}
	\caption{The cones
	$\protect\Cs_2(4)$,
	$\protect\Ccr_2(4)$ and
	$\Ccu_2(4)$
	}
\label{fig:cones}
\end{figure}
The dimension of a cone is defined as the dimension of the linear subspace it generates. Note that the dimensions of our cones are
$$
\dim( \Ccr_n(\sigma^2) )
=
\dim( \Ccu_n(\sigma^2) )
=
n+1, \qquad
\dim( \Cs_n(\sigma^2) )=n.
$$
The following proposition provides convenient representations for the cones  $\Ccr_n(\sigma^2)$ and $\Ccu_n(\sigma^2)$ in terms of the norms that define the underlying polytopes.
\begin{proposition} \label{prop:alt_forms_of_cones}
	For $n\in\N$ and $\sigma >0$ the cones $\Ccr_n(\sigma^2)$ and $\Ccu_n(\sigma^2)$ given in~\eqref{eq:def_Ccr_n} and~\eqref{eq:def_Ccu_n} satisfy
	\begin{align}
	\Ccr_n(\sigma^2) &= \left\{ x=(x_1,\ldots,x_{n+1})\in\R^{n+1} : x_{n+1}\geq \sigma \sum_{i=1}^{n}|x_i| \right\}, \label{eq:Ccr_alt_form} \\
	\Ccu_n(\sigma^2) &= \left\{ x=(x_1,\ldots,x_{n+1})\in\R^{n+1} : x_{n+1}\geq \sigma \max_{1\leq i \leq n }|x_i| \right\}.\label{eq:Ccu_alt_form}
	\end{align}	
\end{proposition}

It is possible (and will be necessary in later applications) to define the cone $\Cs_n(r)$ in the range slightly larger than $r= \sigma^2 > 0$.  For $n\in \N$ and $r > -\frac{1}{n}$ let $u_1,\ldots, u_n$ be vectors in an $N$-dimensional Euclidean space $\R^N$  such that for every $i,j\in \{1,\ldots,n\}$,
\begin{align}
\langle u_i, u_j \rangle = r + \delta_{i,j}, \label{eq:Cs_scalar_products}
\end{align}
where $\delta_{i,j}$ denotes the Kronecker delta.  Note that such vectors exist because the $n\times n$-matrix with entries $r + \delta_{i,j}$ is positive definite for $r > -\frac{1}{n}$, as can be seen from the inequality
$$
\sum_{i=1}^n \sum_{j=1}^n (r + \delta_{i,j}) x_i x_j  = r (x_1+\ldots+x_n)^2  +  (x_1^2 + \ldots + x_n^2) > 0
$$
which is valid for all real $x_1,\ldots,x_n$ and follows from the inequality between the arithmetic and quadratic means.
We define $\Cs_n(r)$ to be the positive hull of such $u_1,\ldots,u_n$. For different choices of $u_1,\ldots, u_n$ we obtain different cones, but all these cones are isometric. Thus, $\Cs_n(r)$ is well defined up to isometry for all $r> -\frac{1}{n}$. The cones $\Cs_n(r)$ have been first introduced and studied by Vershik and Sporyshev~\cite{vershik_sporyshev} under the name contracted ($r>0$) and extended ($r<0$) orthants. For a review of their properties, we refer to~\cite{kabluchko2017absorption}, where these cones are denoted by $C_n(r)$.

In the present paper, the main focus lies on the cones $\Ccr_n(\sigma^2)$ and $\Ccu_n(\sigma^2)$. These cones also can  be characterized by the scalar products of the spanning vectors.
Denoting the vectors spanning $\Ccr_n(\sigma^2)$ by $v_i^+:=\sigma \eee_{n+1}+ \eee_i$ and  $v_i^-:=\sigma \eee_{n+1}- \eee_i$, where $i\in\{1,\ldots,n\}$, we have
\begin{align}
\langle v_i^+, v_j^+ \rangle = \langle v_i^-, v_j^- \rangle = \sigma^2 + \delta_{i,j}, \qquad
\langle v_i^+, v_j^- \rangle = \sigma^2 - \delta_{i,j} \label{eq:Ccr_scalar_products}
\end{align}
for $i,j\in \{1,\ldots,n\}$.
Analogously, denoting the vectors spanning $\Ccu_n(\sigma^2)$ by $v_\eps=\sigma \eee_{n+1}+\sum_{i=1}^{n}\eps_i \eee_i$, where $\eps = (\eps_1,\ldots,\eps_n)\in \{-1,1\}^n$, we have
\begin{align}
\langle v_{\eps}, v_\eta \rangle = \sigma^2+ \langle \eps , \eta \rangle \label{eq:Ccu_scalar_products}
\end{align}
for $\eps, \eta \in \{-1, 1\}^n$. Equations~\eqref{eq:Ccr_scalar_products} and~\eqref{eq:Ccu_scalar_products} explain why we use $\sigma^2$ rather than $\sigma$ as the argument in  $\Ccr_n(\sigma^2)$ and $\Ccu_n(\sigma^2)$.

Knowing these scalar products will be helpful when proving that a certain cone $C$ is isometric to one of the cones above. If $C$ is defined as a positive hull of a finite set $A$ of vectors, to prove the isometry with one of the above cones, it suffices to show that the set $A$ satisfies \eqref{eq:Cs_scalar_products}, \eqref{eq:Ccr_scalar_products} or \eqref{eq:Ccu_scalar_products}.
%It is well known that two cones $C_1=\pos(v^1_1,\ldots,v^1_n)$ and $C_c=\pos(v^2_1,\ldots,v^2_m)$ are isometric if and only if $m=n$ and for every pair $i,j\in\{1,\ldots,n\}$ the scalar products $\langle v^2_i,v^2_j \rangle=\langle v^2_i,v^2_j \rangle$ are equal. So ist das noch falsch. Man bräuchte noch eine Bijektion $\{1,\ldots,n\}\to \{1,\ldots,n\}$\ldots

%When proving that certain convex cones are isometric to one of those above, it will be helpful to know the scalar products of the vectors spanning them. These and alternative forms of two of these cones are given in the following remark.

\section{Angles and intrinsic volumes} \label{sec:cones_angles}
The aim of the present section is to state results on solid angles and conic intrinsic volumes of the cones $\Ccr_n(\sigma^2)$ and $\Ccu_n(\sigma^2)$. The proofs will be given in Section~\ref{sec:proofs_cones_and_angles}.

\subsection{Solid angles}
The solid angle of a cone $C$ is defined as follows. Let $N$ be a random vector with some rotationally invariant distribution on the linear subspace generated by $C$. Then, the \textit{solid angle} of $C$ is defined as
$$
\alpha(C) := \P(N\in C).
$$
As an example of a rotationally invariant distribution, we can take the multivariate standard normal distribution. Since the cones $\Ccr_n(\sigma^2)$ and $\Ccu_n(\sigma^2)$ have the full dimension $n+1$, we can take $N= (\xi_1,\ldots,\xi_{n+1})$, where $\xi_1,\ldots,\xi_{n+1}$ are independent standard normal random variables, and the representation given in Proposition~\ref{prop:alt_forms_of_cones} immediately yields the following
\begin{corollary} \label{cor:angles_cones}
For $n\in \N$ and $\sigma^2>0$,	the solid angles of the cones $\Ccr_n(\sigma^2)$ and $\Ccu_n(\sigma^2)$ are given by
	\begin{align*}
	\alpha\left(\Ccr_n(\sigma^2)\right)&=\P\left( \frac{1}{\sigma} \xi_{n+1}\ge \sum_{j=1}^n|\xi_j| \right), \\
	\alpha\left(\Ccu_n(\sigma^2)\right)&=\P\left( \frac{1}{\sigma} \xi_{n+1}\ge \max_{1\le j\le n}|\xi_j| \right) ,
	\end{align*}
	where $\xi_1,\ldots,\xi_{n+1}$ are i.i.d.\ standard normal random variables.
\end{corollary}
We denote the solid angles above by
\begin{align*}
\gcr_n(\sigma^2):&=\P\left( \frac{1}{\sigma} \xi_{n+1}\ge \sum_{j=1}^n|\xi_j| \right), \\
\gcu_n(\sigma^2):&=\P\left( \frac{1}{\sigma} \xi_{n+1}\ge \max_{1\le j\le n}|\xi_j| \right),
\end{align*}
where $n\in \N$ and $\sigma^2>0$.
Similarly, we can define
$$\gs_n(r) := \alpha (\Cs_n(r))$$
for $r\geq -\frac{1}{n}$. In \cite[Proposition 1.5]{kabluchko2017absorption}, it was shown that
\begin{align}
\gs_n(r)=\P[\eta_1<0,\ldots,\eta_n<0], \label{eq:Cs_alt_form}
\end{align}
where $(\eta_1,\ldots, \eta_n)$ is a Gaussian vector with  zero mean and covariance matrix
$$
\Cov(\eta_i, \eta_j)= \delta_{i,j}-\frac{r}{1+n r}, \qquad i,j\in \{1,\ldots, n\}.
$$
For a review of the properties of the function $\gs_n(r)$ we refer to~\cite{kabluchko2017absorption}. Note that $\gs_n(r)$ coincides with $g_n(-r/(1+nr))$ in the notation of~\cite{kabluchko2017absorption}.  We extend the above definitions to the case $n=0$ by putting $\gcr_0(\sigma^2):=1/2$,  $\gcu_0(\sigma^2):=1/2$, and $\gs_0 (r):=1$.

\subsection{Polar cones}
For a polyhedral cone $C\subset \R^n$, its \textit{polar cone} is defined by
$$
C^\circ := \{x\in \R^n \colon \langle x,y\rangle \leq 0 \text{ for all } y\in C\}.
$$
It is known that $C^{\circ \circ} = C$.  
The following proposition may be viewed as a cone version of the classical polarity relation between  crosspolytope and cube. 
\begin{proposition} \label{prop:duality_cones}
For every $\sigma^2>0$ and $n\in \N$, we have
%the polar cone of  $\Ccu_n(\sigma^2)$ coincides with  $-\Ccr_n(\frac{1}{\sigma^2})$  and vice versa, the polar cone of $\Ccr_n(\sigma^2)$ coincides with  $-\Ccu_n(\frac{1}{\sigma^2})$.
$$
\left(\Ccu_n(\sigma^2)\right)^{\circ} = -\Ccr_n\left(\frac{1}{\sigma^2}\right),
\qquad
\left(\Ccr_n(\sigma^2)\right)^{\circ} = -\Ccu_n\left(\frac{1}{\sigma^2}\right).
$$
\end{proposition}
While cones associated with the crosspolytope and the cube are polar to each other by the above proposition, the cones associated with the simplex are self-polar in the following sense: If $C\subset \R^n$ is an isometric copy of $\Cs_n(r)$ (so that $C$ has full dimension in $\R^n$), then $C^\circ$ is isometric to $\Cs_n(-r/(1+nr))$; see Proposition 2.2 in~\cite{kabluchko2017absorption}.

\subsection{Angles of the crosspolytope}
Given  a polytope $P$ and a face $F$, the tangent cone $T_F(P)$ of $P$ at $F$ is defined as the positive hull of the set $P-f_0$, where $f_0$ is some fixed point in the relative interior of $F$. The internal solid angle of $P$ at $F$ is the angle of this tangent cone. The normal (or external) solid angle of $P$ at $F$ is defined as the angle of the polar of the tangent cone.
The next proposition expresses the internal and the external solid angles at the faces of the regular crosspolytope through the quantities $\gcr_n(\sigma^2)$ and $\gcu_n(\sigma^2)$ introduced above.

\begin{proposition} \label{prop:inner_outer_angle}
	For $n\in \N$ and $k\in \{0,\ldots, n-1\}$ let $F$ be a $k$-dimensional face of the $n$-dimensional crosspolytope $P_n$. The internal solid angle of $P_n$ at $F$ equals
	\begin{align}
		\gcr_{n-k-1}\left(\frac{1}{k+1}\right) = \P\left( \sqrt{k+1} \xi_{n-k}\ge \sum_{j=1}^{n-k-1}|\xi_j| \right) . \label{eq:crosspoly_inner_solid_angle}
	\end{align}
	The normal solid angle of $P_n$ at $F$ equals
	\begin{align}
		\gcu_{n-k-1}(k+1)=\P\left( \xi_{n-k}\ge \sqrt{k+1} \max_{1\le j\le n-k-1}|\xi_j| \right) . \label{eq:crosspoly_outer_solid_angle}
	\end{align}
\end{proposition}
For the angles of the regular simplex, similar expressions in terms of $\gs_n(r)$ are possible; see, e.g.\ \cite[Proposition~1.2]{kabluchko2017absorption}.

\subsection{Conic intrinsic volumes}
The $k$\textsuperscript{th} \textit{conic intrinsic volume} of an $m$-dimensional polyhedral cone $C\subset \R^m$ is defined by
$$
\upsilon_k (C) := \sum_{F\in \cF_k(C)} \alpha(F) \alpha(N_F(C)),  \qquad k\in \{0,\ldots,m\},
$$
where $N_F(C) = C^\circ \cap (\lin F)^\bot$ is the face of $C^\circ$ corresponding to $F$ via the polar duality.
We refer to~\cite[Section~6.5]{SW08} and~\cite{amelunxen}, \cite{amelunxen_comb} for an extensive account of the properties of conic intrinsic volumes.
\begin{theorem} \label{thm:intrinsic_volumes}
	For $n\in\N$ and $k\in \{1,\ldots, n+1\}$,
	the $k$\textsuperscript{th} conic intrinsic volume of the cone $\Ccu_n(\sigma^2)$ is
	\begin{align}
		\upsilon_k\left( \Ccu_n(\sigma^2) \right)
		=
		2^{n-k+1} \binom{n}{k-1} \gcu_{k-1}\left(\sigma^2+n-k+1\right) \gs_{n-k+1}\left( \frac{1}{\sigma^2} \right). \label{eq:Ccu_intrinsic_volume}
	\end{align}
	For $k\in \{0,\ldots, n\}$, the $k\textsuperscript{th}$ conic intrinsic volume of the cone $\Ccr_n(\sigma^2)$ is
	\begin{align}
		\upsilon_k\left( \Ccr_n(\sigma^2) \right)
		=
		2^{k} \binom{n}{k} \gcu_{n-k}\left(\frac{1}{\sigma^2}+k\right) \gs_{k}\left( \sigma^2 \right) . \label{eq:Ccr_intrinsic_volume}
	\end{align}
	The exceptional cases are
	$$
		\upsilon_0(\Ccu_n(\sigma^2))= \gcr_n\left(\frac{1}{\sigma^2}\right), \qquad \upsilon_{n+1}(\Ccr_n(\sigma^2))= \gcr_n(\sigma^2).
	$$
%	The $0$\textsuperscript{th} intrinsic volumes of these cones are
%	$$
%	\upsilon_0(\Ccu_n(\sigma^2))= \gcr_n\left(\frac{1}{\sigma^2}\right),
%	\qquad
%	\upsilon_0(\Ccr_n(\sigma^2))= \gcu_n\left(\frac{1}{\sigma^2}\right).
%	$$
%	Their $(n+1)$\textsuperscript{th} intrinsic volumes are
%	$$
%	\upsilon_{n+1}(\Ccu_n(\sigma^2))= \gcu_n(\sigma^2),
%	\qquad
%	\upsilon_{n+1}(\Ccr_n(\sigma^2))= \gcr_n(\sigma^2).
%	$$
\end{theorem}

\section{Absorption probabilities}\label{sec:absorption_probab}
In this section we give some applications of Theorem~\ref{thm:intrinsic_volumes} to the determination of absorption probabilities of certain random polytopes.

\subsection{Gaussian projections of regular polytopes}
Let $X_1,\ldots, X_n$ be independent standard normal random points in $\R^d$.
The \textit{Gaussian polytope} is defined as the convex hull of these random points, i.e.\
$$
\Pgau_{n,d} := \conv \{X_1,\ldots, X_n\}.
$$
Similarly, the \textit{symmetric Gaussian polytope} $\Psym_{n,d}$ is defined as the convex hull of these points along with their negatives, i.e.\
$$
\Psym_{n,d} := \conv \{X_1, -X_1, X_2,\ldots, X_n, -X_n\}.
$$
Finally, the \textit{Gaussian zonotope} is the Minkowski sum of $n$ Gaussian intervals, i.e.\ with $X_1,\ldots, X_n$ as before,
$$
\Pzono_{n,d} := \sum_{i=1}^{n} \conv\{X_i, -X_i\} = \left\{\sum_{i=1}^n \lambda_i X_i: \lambda_1,\ldots,\lambda_n \in [-1,1] \right\}.
$$
%\begin{center}
	\begin{figure}[ht]
%		\centering
		\begin{minipage}[c]{0.3\textwidth}
			\includegraphics[width=\textwidth]{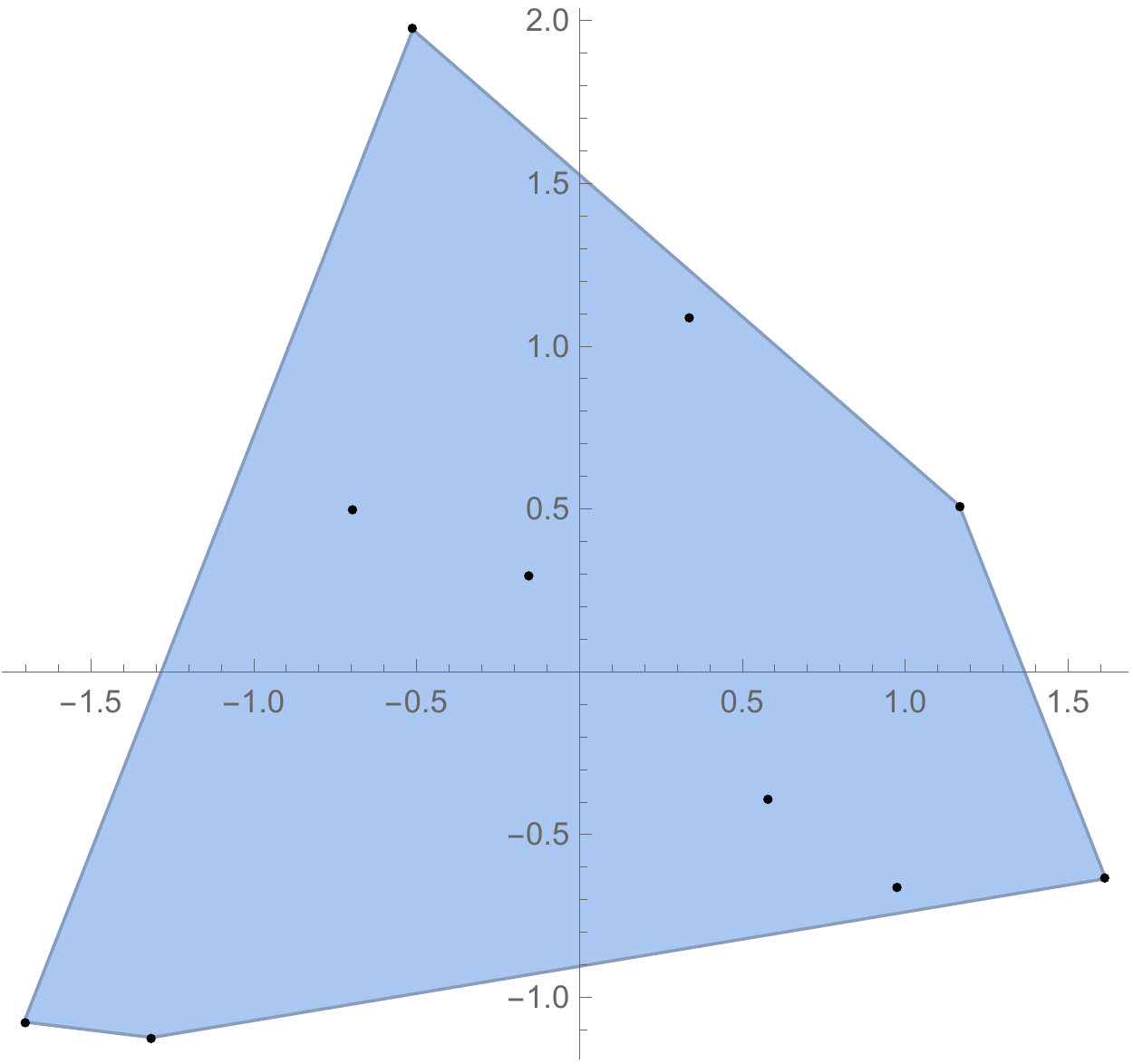}
			\vspace*{21pt}
		\end{minipage}
		\begin{minipage}[c]{0.3\textwidth}
			\includegraphics[width=\textwidth]{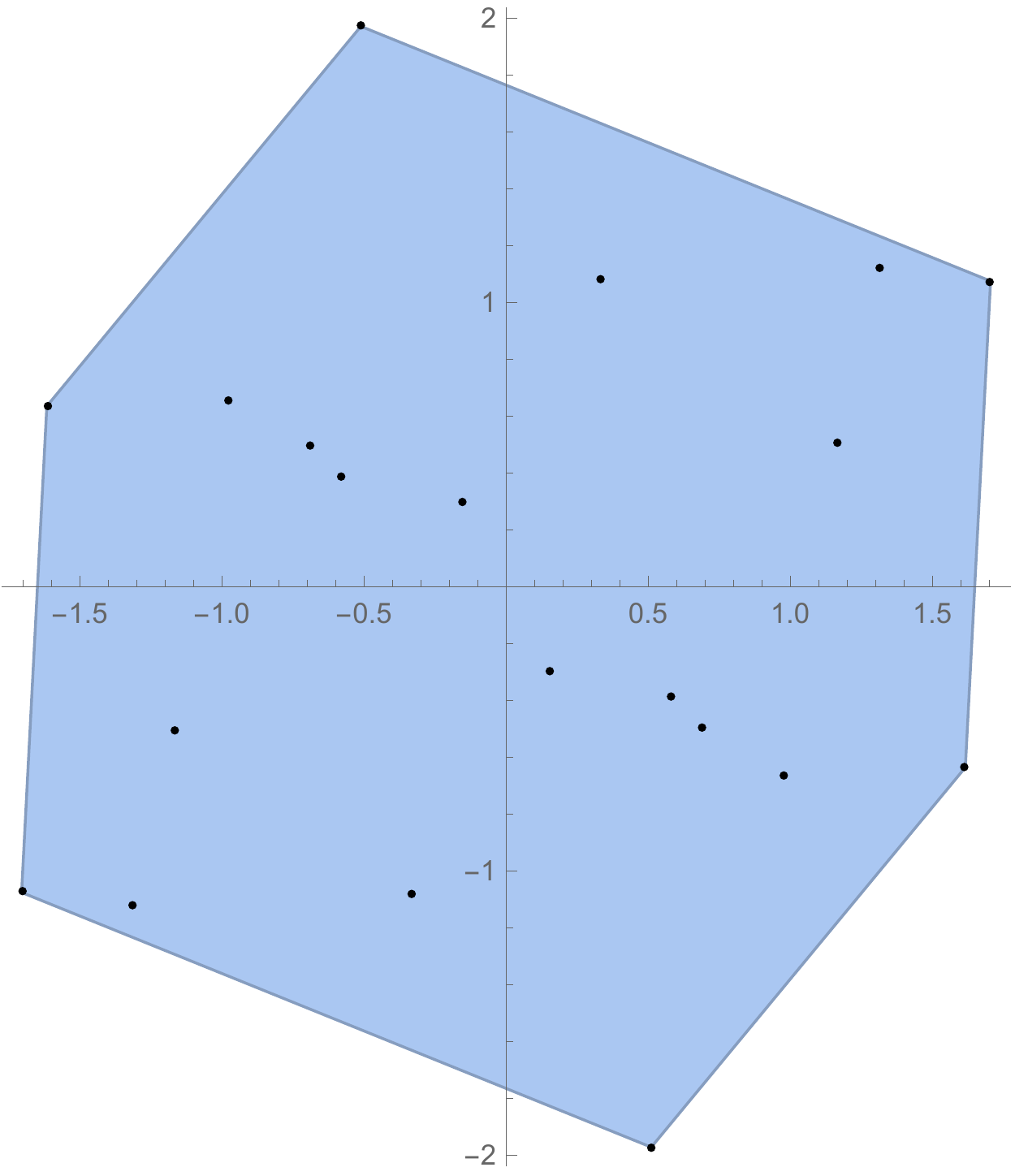}
		\end{minipage}
		\begin{minipage}[c]{0.3\textwidth}
			\includegraphics[width=\textwidth]{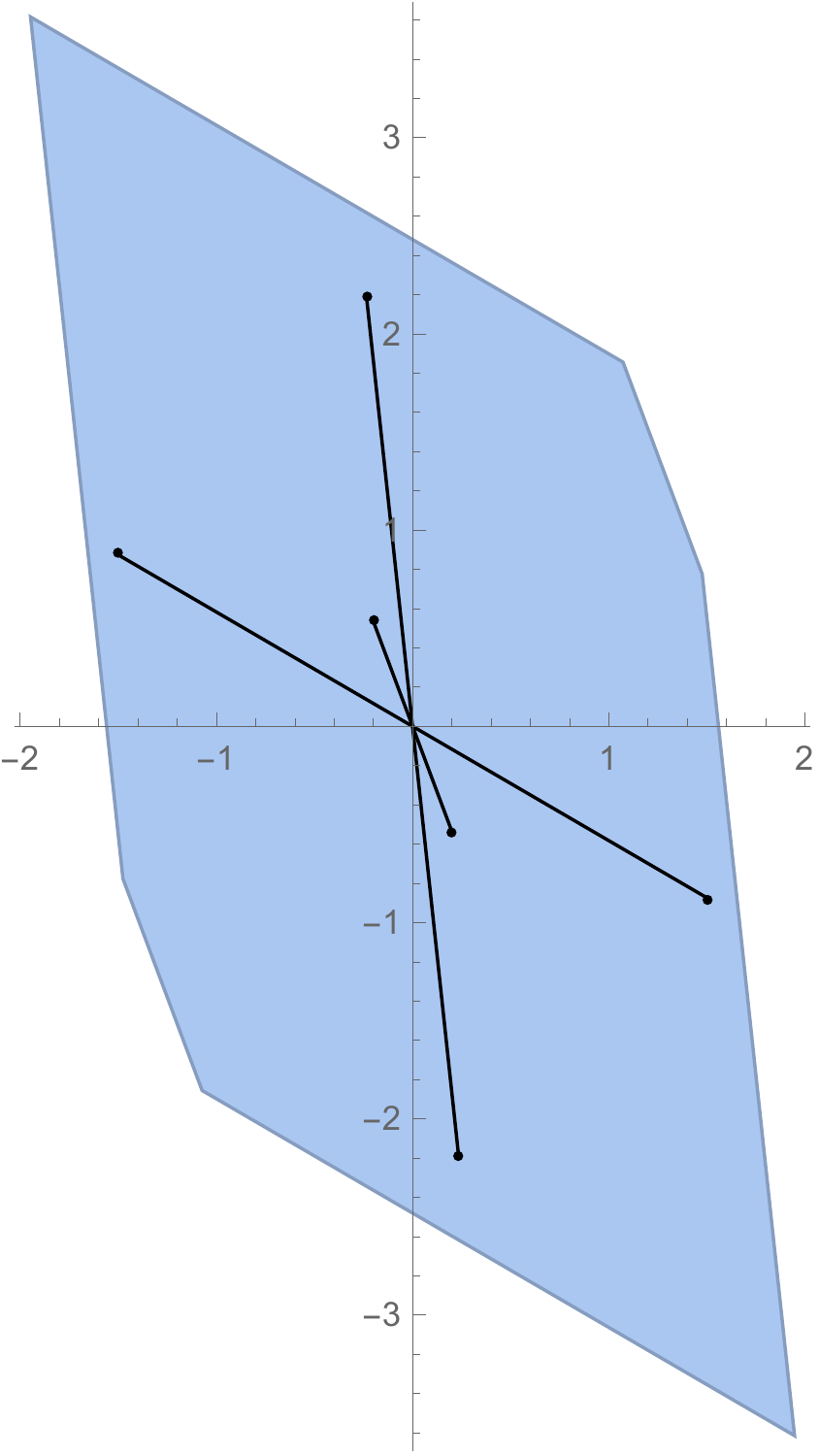}
		\end{minipage}
		\caption{Examples of a Gaussian polytope, a symmetric Gaussian polytope and a Gaussian zonotope in dimension 2}
	\end{figure}
%\end{center}
These three random polytopes are related to the three regular polytopes via the notion of Gaussian projection. Let $P$ be any (deterministic) polytope in $\R^n$. Let also $X$ be a $d\times n$-matrix whose entries are i.i.d.\ standard Gaussian random variables. The columns of $X$ can be identified with the random vectors $X_1,\ldots,X_n$ introduced above. We may consider $X: \R^n \to\R^d$ as a random linear map.  Then, the \textit{Gaussian projection} of $P$ (see, e.g., \cite{vershynin_book}) is defined as the random polytope $X P = \{Xp : p\in P\}\subset \R^d$. It is now easy to check that by taking $P$ to be the regular simplex $\conv \{e_1,\ldots,e_n\}$, the regular crosspolytope $\conv \{\pm e_1,\ldots,\pm e_n\}$ and the cube $[-1,1]^n$, we recover the random polytopes $\Pgau_{n,d}$, $\Psym_{n,d}$, $\Pzono_{n,d}$ as the corresponding Gaussian projections. % (just identify the columns of $X$ with $X_1,\ldots,X_n$).
	
The associated absorption probability is the probability of the event that a deterministic point $x\in\R^d$ is contained in the polytope $\Pgau_{n,d}$ (and similarly for $\Psym_{n,d}$ and  $\Pzono_{n,d}$). Since the standard Gaussian distribution is invariant under rotations, this probability depends on $x$ only by its Euclidean norm $|x|$. Instead of the absorption probability itself it is convenient to analyse the probabilities of non-absorption $\fgau_{n,d},\fsym_{n,d}, \fzono_{n,d} :[0,\infty) \to [0,1]$ defined by
\begin{align*}
\fgau_{n,d} (|x|) := \P(x\notin \Pgau_{n,d}),
\quad
\fsym_{n,d} (|x|) := \P(x\notin \Psym_{n,d}),
\quad
\fzono_{n,d} (|x|) := \P(x\notin \Pzono_{n,d}).
\end{align*}

An expression for the non-absorption probability $\fgau_{n,d}$
%(or, more precisely, for certain Laplace-type transform thereof)
was provided in~\cite[Theorem 1.2]{kabluchko2017absorption}.  The aim of this section is to give similar expressions for the non-absorption probabilities $\fsym_{n,d}$ and $\fzono_{n,d}$ of symmetric Gaussian polytopes and Gaussian zonotopes.

\subsection{Absorption probabilities for polytopes spanned by Gaussian points}
%For independent $d$-dimensional standard normal distributed random variables $X_1,\ldots,X_n$ let $\Psym_{n,d}=\conv (\pm X_1,\ldots,\pm X_n)$ be an $n$-dimensional symmetric Gaussian polytope and $\Pzono_{n,d}:= \sum_{k=1}^{n} \conv\{X_k, -X_k\}$ be an $n$-dimensional Gaussian zonotope.	
%	For $r \in [0,\infty)$
%	$$
%	\fsym_{n,d}(r):=\P(r \eee_1 \notin \Psym_n)
%	$$
For a $d$-dimensional standard normal random vector $X \sim \Normal^d(0,1)$ independent of $X_1,\ldots,X_n$ and $\sigma>0$ we define
	\begin{align*}
	\psym_{n,d}(\sigma^2):&=\P ( \sigma X \notin \Psym_{n,d}),\\
	\pzono_{n,d}(\sigma^2):&=\P ( \sigma X \notin \Pzono_{n,d}).
	\end{align*}
	These probabilities differ from the non-absorption probabilities, because here the point $\sigma X$ is random. Calculating them will be a first step in determining the absorption probability because, as the following proposition states, there is a connection between these two functions.
The proposition is based on \cite[Corollary 1.1]{kabluchko2017absorption}.
\begin{proposition} \label{prop:p_Laplace_trafo_of_f}
Let $P_n\subset \R^d$ be a Gaussian polytope, a symmetric Gaussian polytope or a Gaussian zonotope generated by $n$ independent Gaussian points $X_1,\ldots,X_n$. Its non-absorption probability $f^{P_n}\colon [0,\infty) \to [0,1]$ is defined by $f^{P_n}(|x|) :=  \P(x\notin P_n)$. Then,
		$$
		\int_{0}^{\infty} f^{P_n}\left(\sqrt{2 u}\right) u^{\frac{d}{2}-1} \eee^{-\lambda u} \dd u
		=
		\Gamma\left(\frac{d}{2}\right) \lambda^{-\frac{d}{2}} p^{P_n}\left( \frac{1}{\lambda} \right)
		$$
		for every $\lambda >0$. Here,  $p^{P_n}(\sigma^2):= \P ( \sigma X \notin P_n)$.
	\end{proposition}
	\begin{proof}%[proof of Proposition \ref{prop:p_Laplace_trafo_of_f}]
		Let $X\sim \Normal^d(0,1)$ be independent of $X_1,\ldots,X_n$ as in the definition of $p^{P_n}$. Its Euclidean norm $|X|$ has $\chi$ distribution with $d$ degrees of freedom. Conditioning on the event $|X|=r$ and integrating over $r>0$ we obtain
		$$
		p^{P_n}(\sigma^2)
		=
		\P\left[ \sigma X \notin \conv \{X_1,\ldots,X_n\} \right]
		=
		\int_{0}^{\infty} f^{P_n}(\sigma r) \frac{2^{1-\frac{d}{2}}}{\Gamma\left(\frac{d}{2}\right)} r^{d-1} \eee^{-\frac{r^2}{2}} \dd r.
		$$
		Substituting $\sigma r = \sqrt{2u}$, taking into account that $\dd r = \frac{1}{\sigma} \frac{\dd u}{\sqrt{2u}}$ and writing $\lambda = \frac{1}{\sigma^2}$, we obtain
		$$
		p^{P_n}\left(\frac 1 \lambda\right)= \frac{ \lambda^{\frac{d}{2}}}{\Gamma(\frac{d}{2})} \int_{0}^{\infty} f^{P_n}(\sqrt{2u}) u^{\frac{d}{2}-1}\eee^{-\lambda u} \dd u,
		$$
		thus completing the proof.
	\end{proof}

\subsection{Expressions for  \texorpdfstring{$p^{P_n}$}{p\^P\_n}}
We will now provide explicit expressions for $p^{P_n}$.
\begin{theorem}\label{thm:p_symm}
	For all $n,d\in\N$ such that $n\geq d$ the probability just defined satisfies
	$$
	\psym_{n,d}(\sigma^2)=\P ( \sigma X \notin \Psym_{n,d})=2(\bsym_{n,d-1}(\sigma^2)+\bsym_{n,d-3}(\sigma^2)+\ldots)
	$$
	with
$$
\bsym_{n,k}(r)
:=
\upsilon_k(\Ccr_n(r))
=
\begin{cases}
2^{k} \binom{n}{k} \gcu_{n-k}\left(\frac{1}{r}+k\right) \gs_{k}(r), & \text{ if } k\in \{0,\ldots,n\},\\
\gcr_n(r), &\text{ if } k=n+1,
\end{cases}
%\qquad
%k\in \{0,\ldots,n+1\},
%=2^{k} \binom{n}{k} \gcu_{n-k}\left(\frac{1}{r}+k\right) \gs_{k}\left( r \right).
$$
as in Theorem \ref{thm:intrinsic_volumes} and $\bsym_{n,k}(r) = 0$ for $k\notin \{0,\ldots,n+1\}$.
\end{theorem}

\begin{theorem} \label{thm:p_zono}
For all $n,d\in\N$ such that $n\geq d$ we have
	$$
		\pzono_{n,d}(\sigma^2) = \P(\sigma X \notin \Pzono_{n,{d}})
		=
		2(\bcu_{n,d-1}(\sigma^2)+\bcu_{n,d-3}(\sigma^2)+\ldots),
	$$
	where
	$$
		\bcu_{n,k}(r)
		:=	
		\upsilon_k\left( \Ccu_n(r) \right)
		=
		2^{n-k+1} \binom{n}{k-1} \gcu_{k-1}\left(r+n-k+1\right) \gs_{n-k+1}\left( \frac{1}{r} \right)
	$$
for $k\in \{0,\ldots,n+1\}$ and $\bcu_{n,k}(r) = 0$ for $k\notin \{0,\ldots,n+1\}$.
\end{theorem}

Combining Theorems~\ref{thm:p_symm} and~\ref{thm:p_zono} with Proposition \ref{prop:p_Laplace_trafo_of_f} we arrive at the following
\begin{corollary} \label{Cor:like_1.3}
	For every $\lambda>0$ and for $n,d\in \N$ such that $n\geq d$ the absorption probabilities $\fsym_{n,d}$ and $\fzono_{n,d}$
	%and the probabilities $\psym_{n,d}$ and $\pzono_{n,d}$
	satisfy the equations
	\begin{align}
		\int_{0}^{\infty} \fsym_{n,d}\left(\sqrt{2 u}\right) u^{\frac{d}{2}-1} \eee^{-\lambda u} \dd u
		&=
		2 \Gamma\left(\frac{d}{2}\right) \lambda^{-\frac{d}{2}} \left(\bsym_{n,d-1}\left(\frac{1}{\lambda}\right)+\bsym_{n,d-3}\left(\frac{1}{\lambda}\right)+\ldots\right),
		\\
		\int_{0}^{\infty} \fzono_{n,d}\left(\sqrt{2 u}\right) u^{\frac{d}{2}-1} \eee^{-\lambda u} \dd u
		&=
		2 \Gamma\left(\frac{d}{2}\right) \lambda^{-\frac{d}{2}} \left(\bzono_{n,d-1}\left(\frac{1}{\lambda}\right)+\bzono_{n,d-3}\left(\frac{1}{\lambda}\right)+\ldots\right).
	\end{align}
%	with, as before,
%	\begin{align*}
%		\bsym_{n,k}(r):&= \upsilon_k\left(\Ccr_n(r)\right)%=2^{k} \binom{n}{k} \gcu_{n-k}\left(\frac{1}{r}+k\right) \gs_{k}\left( r \right)
%		\\
%		\bcu_{n,k}(r)
%		:&=	
%		\upsilon_k\left( \Ccu_n(r) \right)
%		=
%		2^{n-k+1} \binom{n}{k-1} \gcu_{k-1}\left(r+n-k+1\right) \gs_{n-k+1}\left( \frac{1}{r} \right).
%	\end{align*}
\end{corollary}

\subsection{Absorption probabilities in dimension \texorpdfstring{$d=2$}{d=2}}
In dimension $d=2$ the equalities in Corollary \ref{Cor:like_1.3} simplify to
\begin{align*}
	\int_{0}^{\infty} \fsym_{n,2}\left(\sqrt{2 u}\right) \eee^{-\lambda u} \dd u
	&=
	\frac{2}{\lambda} \bsym_{n,1}\left(\frac{1}{\lambda}\right),
	\\
	\int_{0}^{\infty} \fzono_{n,2}\left(\sqrt{2 u}\right) \eee^{-\lambda u} \dd u
	&=
	\frac{2}{\lambda} \bzono_{n,1}\left(\frac{1}{\lambda}\right).
\end{align*}
Using that $\gs_1(r)=\frac{1}{2}$ for every $r > -\frac{1}{n}$ and $\gcu_0( r )= \frac{1}{2}$ for every $r > 0$ we have
\begin{align*}
\bsym_{n,1}\left(\frac{1}{\lambda}\right)
&=
\upsilon_1\left(\Ccr_n \left( \frac{1}{\lambda} \right) \right)
=
n\cdot \gcu_{n-1}(\lambda +1),
\\
\bzono_{n,1}\left(\frac{1}{\lambda}\right)
&=
\upsilon_1\left(\Ccu_n \left( \frac{1}{\lambda} \right) \right)
=
2^{n-1} \gs_{n}\left( \lambda \right).
\end{align*}
Thus,
\begin{align}
	\int_{0}^{\infty} \fsym_{n,2}\left(\sqrt{2 u}\right) \eee^{-\lambda u} \dd u
	&=
	\frac{2 n}{\lambda} \gcu_{n-1}(\lambda +1)
	=
	\frac{2 n}{\lambda} \P\left( \frac{1}{\sqrt{\lambda +1}} \xi_{n}\ge \max_{1\le j\le {n-1}}|\xi_j| \right), \label{eq:rhs_laplace_transf}
	\\
	\int_{0}^{\infty} \fzono_{n,2}\left(\sqrt{2 u}\right) \eee^{-\lambda u} \dd u
	&=
	\frac{2^n}{\lambda} \gs_{n}\left( \lambda \right).
	%=
	%\frac{2^n}{\lambda} g_n^\text{Kabluchko, Zaporozhets}\left( - \frac{\lambda}{1+n \lambda} \right).
\end{align}
So, to calculate $\fsym_{n,2}$ and $\fzono_{n,2}$ it is sufficient to invert the Laplace transforms on the right-hand side. This can be done, for example, by using the Bromwich integral~\cite[Chapter~4]{schiff_book}.
In the case of the symmetric Gaussian polytope, a more explicit inversion is possible and stated in the  following result.

\begin{theorem}
Let $\xi, \xi_1,\ldots,\xi_n$ be independent standard normal random variables. Define $L_n:= \max\{|\xi_1|,\ldots,|\xi_n|\}$. Then, for all $u>0$, we have
$$
\fsym_{n,2} (\sqrt{2u}) = \P\left(\frac {L_n^2 + \xi^2} 2 \leq u\right) +  \frac{\dd}{\dd u} \P\left(\frac {L_n^2 + \xi^2} 2 \leq u\right).
$$
\end{theorem}
%\subsubsection{Inverting the Laplace Transform for the Symmetric Gaussian Polytope}
\begin{proof}
The right-hand side of~\eqref{eq:rhs_laplace_transf} can be written as
\begin{align*}
	\frac{2 n}{\lambda} \P\left( \frac{1}{\sqrt{\lambda +1}} \xi_{n}\ge \max_{1\le j\le {n-1}}|\xi_j| \right)
	=
	\frac{2 n}{\lambda} \frac{1}{\sqrt{2 \pi}} \int_{0}^{\infty} \eee^{-\frac{s^2}{2}} \left(2\Phi\left( \frac{s}{\sqrt{\lambda+1}} \right)-1\right)^{n-1}\dd s,
\end{align*}
where $\Phi$ is the cumulative distribution function of the standard normal distribution $\Normal(0,1)$.
Substituting $\frac{s^2}{2}= t(\lambda+1)$, so that $s=\sqrt{2t (\lambda+1)}$ and $\dd s = \sqrt{\frac{\lambda+1}{2t}} \dd t$, we rewrite the right-hand side as
\begin{align}
\lefteqn{\frac{2 n}{\lambda} \frac{1}{\sqrt{2 \pi}} \int_{0}^{\infty} \eee^{-\lambda t} \eee^{-t} (2\Phi(\sqrt{2t})-1)^{n-1} \sqrt{\frac{\lambda+1}{2t}} \dd t}\notag
\\
	&=
	 \frac{\sqrt{\lambda+1}}{\lambda} \int_{0}^{\infty} \eee^{-\lambda t} n \frac{2 \eee^{-t}}{\sqrt{2\pi}} (2\Phi(\sqrt{2t})-1)^{n-1} \frac{1}{\sqrt{2t}} \dd t \notag\\
	&=
	\frac{\sqrt{\lambda+1}}{\lambda} \int_{0}^{\infty} \eee^{-\lambda t} \frac{\dd}{\dd t} \left((2\Phi(\sqrt{2t})-1)^n\right) \dd t. \label{eq:Laplace_trafo}
\end{align}
%Kabluchko and Zaporozhets state in \cite[Theorem 1.22]{kabluchko2017absorption} that
Now, $(2\Phi(\sqrt{2t})-1)^n$ is the distribution function of $\frac 12 L_n^2$, where we recall that $L_n:= \max\{|\xi_1|,\ldots,|\xi_n|\}$ and the $\xi_i$'s are independent standard normal random variables.

On the other hand, the inverse Laplace transform of $\frac{\sqrt{\lambda +1}}{\lambda}$ is
$$
2\Phi(\sqrt{2t})- 1 + \frac{\eee^{-t}}{\sqrt{\pi t}} = F_{\frac{\xi^2}{2}}(t)+f_{\frac{\xi^2}{2}}(t) ,
$$
where $F_{\frac{\xi^2}{2}}$ is the distribution function and $f_{\frac{\xi^2}{2}}$ the density of $\frac{\xi^2}{2}$, with $\xi$ being standard normal.

Thus the inverse Laplace transform of \eqref{eq:Laplace_trafo} is the convolution of $F_{\frac{\xi^2}{2}}+f_{\frac{\xi^2}{2}}$ and $f_{\frac 12 L_n^2}$, where $f_{\frac 12 L_n^2}$ is the density function of $\frac 12 L_n^2$. It follows from~\eqref{eq:rhs_laplace_transf} that
\begin{align*}
\fsym_{n,2} (\sqrt{2u})
&=
\int_0^u F_{\frac{\xi^2}{2}}(t) f_{\frac 12 L_n^2} (u-t) \dd t +
+
\int_0^u f_{\frac{\xi^2}{2}}(t) f_{\frac 12 L_n^2} (u-t) \dd t\\
&=
\P\left(\frac {L_n^2 + \xi^2} 2 \leq u\right) +  \frac{\dd}{\dd u} \P\left(\frac {L_n^2 + \xi^2} 2 \leq u\right),
\end{align*}
which completes the proof.
\end{proof}

\section{Random sections of regular polytopes}\label{sec:random_sections}
In \cite{lonke} Lonke investigated the asymptotics of the  expected number of $j$-faces of the intersection of the $n$-cube $[-1,1]^n$ and a random $k$-dimensional linear subspace of $\R^{n}$ chosen uniformly from the Grassmannian $Gr(k,\R^n)$, i.e.\ the set of all $k$-dimensional linear subspaces of $\R^n$.
%He gives answers for several special cases, but does not\ldots
With the methods we used above we can give explicit expressions for these expected numbers not only for cubes, but also for crosspolytopes and simplices.
\\
In Section \ref{subsec_random_section_cones} we will determine the probabilities that fixed faces of our polytopes get intersected by the random linear space. This will be an auxiliary result for Section \ref{subsec:random_sections_polytopes}, in which we will state the expressions we are interested in. %In Subsection \ref{subsec_random_section_asymptotic} we will derive an asymptotic formula for \ldots

%\subsection{General Position and Connections Between Random Sections of Cones and of Polytopes} \label{subsec_tool}
Let $L$ a random $(n-l)$-dimensional linear subspace of $\R^n$ having the uniform distribution on $\Gr(n-l,\R^n)$. Here, $l \in \{1,\ldots, n-1\}$ is the codimension of $L$.
Let the random variable $\phi^P(j,n-l,n)$ be the number of $j$-faces of the intersection of $L$ and an $n$-dimensional convex polytope $P\subset \R^n$ which contains the origin in its interior.

\begin{proposition} \label{Prop:Random_sections_first_subsection}
For $n\in \{2,3,\ldots\}$, $l \in \{1,\ldots, n-1\}$ and $j\in \{0, \ldots , n-l-1\}$.
\begin{align}
	\phip(j,n-l,n)= \sum_{B\in \cF_{j+l}(P)} \indi{B \cap L \neq \emptyset} \text{ almost surely.} \label{eq:random_sections_first_identity}
\end{align}
\end{proposition}
In other words, almost surely the number of $j$-faces of the intersection $L\cap P$ is equal to the number of $(j+l)$-faces of $P$ that have non-empty intersection with $L$.
Equation~\eqref{eq:random_sections_first_identity} looks natural because it is reasonable that with probability $1$ the $j$-faces of $L \cap P$ can be obtained as the intersections of $L$ and the $(j+l)$-dimensional faces of $P$.
The proof of this fact is surprisingly difficult and will be given in Section~\ref{sec:proof_intersection}.
In \cite[(3.1)]{grunbaum} a similar result is stated (without proof) for random projections instead of random sections.

An immediate consequence for regular polytopes $P$ is
\begin{align}
\E \phip(j,n-l,n)=\E \sum_{B\in\cF_{j+l}(P)} \indi{B \cap L \neq \emptyset}
=
\# \cF_{j+l}(P) \cdot \P (B\cap L \neq \emptyset), \label{eq:random_sections_first_identity_exp}
\end{align}
which is true for every $(j+l)$-face $B$.

It should be stressed that, as already observed by Lonke~\cite{lonke}, there is a duality between the number of faces in random intersections as above and the number of faces in random projections of a dual polytope. Namely, the expected number of $j$-faces of a random $k$-dimensional projection of an $n$-dimensional polytope $P$ containing the origin in its interior coincides with the expected number of $(k-j-1)$-faces of the intersection of the dual polytope $P^\circ$ with a random $k$-dimensional linear subspace passing through the origin. The expected face numbers of a random projection of a polytope has been expressed through its internal and external angles in the work of Affentranger and Schneider~\cite{affentranger_schneider}. Asymptotic questions were studied in~\cite{vershik_sporyshev,boroczky,donoho_tanner_ProcAcadUSA,donoho_tanner_JAMS,donoho_tanner_DiscrCompGeom}.

\subsection{Probabilities that fixed faces get intersected} \label{subsec_random_section_cones}
Our purpose is to apply Equation \eqref{eq:random_sections_first_identity_exp} to random sections of the three kinds of regular polytopes we are analysing. Essentially, we need to derive  expressions for the quantity $\P(B\cap L \neq \emptyset)$ for a fixed face of dimension $d=j+l$, which is the second factor in \eqref{eq:random_sections_first_identity_exp}. These expressions are given in the following three lemmas.
For $l\in\{1,\ldots,n-1\}$ we recall that  $L$ is a linear subspace of $\R^n$ that has codimension $l$ and is chosen from the set of all such subspaces uniformly at random.
\begin{proposition} \label{prop:Prob_S_cap_F_notempty_Cube}
	%Fix $n\in \N$, $d \in \{1,\ldots,n-1\}$.
	For $d\in \{l,\ldots, n-1\}$ the probability of the event that $L$ intersects a fixed $d$-face $B$ of the cube $[-1,1]^n$ is
	$$
	\P(L\cap B \neq \emptyset )
	=
	2 \left(
	\upsilon_{l+1} (\Ccu_d(n-d)) +\upsilon_{l+3}(\Ccu_d(n-d))+\ldots
	\right).
	$$
\end{proposition}

\begin{proposition} \label{prop:Prob_S_cap_F_notempty_crosspoly}
	%Fix $n\in \N$, $k,d \in \{1,\ldots,n-1\}$.
	For $d\in \{l,\ldots, n-1\}$ the probability of the event that $L$ intersects a fixed $d$-face $B$ of the crosspolytope $\conv \{\eee_1, -\eee_1,\eee_2,-\eee_2,\ldots, \eee_n,-\eee_n\}$ is
	$$
	\P(L\cap B \neq \emptyset )
	=
	2^{-d} \left(
	\binom{d}{l} + \binom{d}{l+1} + \ldots
	\right).
%	2 \left(
%	\upsilon_{l-1} \left(\Cs_d\left(\frac{1}{d+1}\right)\right) +\upsilon_{l-3}\left(\Cs_d\left(\frac{1}{d+1}\right)\right)+\ldots
%	\right).
	$$
\end{proposition}

\begin{proposition} \label{prop:Prob_S_cap_F_notempty_simplex}
	Let $P_{n} \subset \R^n$ be an $n$-dimensional regular simplex centred at the origin, whose edges have unit length.
	For $d\in \{l,\ldots, n-1\}$ the probability of the event that $L$ intersects a fixed $d$-face $B$ of $P_{n}$ is
	$$
	\P(L\cap F \neq \emptyset )
	=
	2\left(\upsilon_{l+1} \left(\Cs_{d+1} \left( -\frac{1}{n+1} \right)\right)+\upsilon_{l+3}\left(\Cs_{d+1} \left( -\frac{1}{n+1} \right)\right)+\ldots \right).
	$$
\end{proposition}

\subsection{Random sections of regular polytopes} \label{subsec:random_sections_polytopes}
We are now in position to state formulas for the expected number of faces of regular polytopes intersected by a random linear subspace.
\begin{theorem} \label{thm:random_sections_polytopes}
	Fix integers $n > k >j\geq 0$ and let $L$ be a random linear subspace of $\R^n$ that has dimension $k$ and is chosen from the set of all such subspaces uniformly. Denote its codimension by $l:=\codim(L)=n-k$.
	Let $\phicu(j,k,n)$, $\phicr(j,k,n)$ and $\phis(j,k,n)$ be the number of $j$-faces of the intersection of $L$ and respectively the cube $[-1,1]^n$, the crosspolytope $\conv\{\eee_1, -\eee_1,\eee_2,-\eee_2,\ldots, \eee_n,-\eee_n\}$ or the $n$-dimensional simplex in $\R^n$ centred at the origin.
	The expectations of these random variables are given by:
	\begin{align*}
		\E \phicu(j,k,n)&= 2^{k-j+1} \binom{n}{n-k+j} \cdot \\
		&\quad \cdot
		\left(
		\upsilon_{n-k+1} (\Ccu_{n-k+j}(k-j)) +\upsilon_{n-k+3}(\Ccu_{n-k+j}(k-j))+\ldots
		\right),
		\\
		\E\phicr(j,n-l,n)
		&=
		2  \binom{n}{j+l+1} \left(
		\binom{j+l}{l} + \binom{j+l}{l+1} + \ldots
		\right),
		\\
		\E\phis(j,n-l,n)&=
		2\binom{n+1}{j+l+1} \cdot \\
		&\cdot \left(\upsilon_{l+1} \left(\Cs_{j+l+1} \left( -\frac{1}{n+1} \right)\right)+\upsilon_{l+3}\left(\Cs_{j+l+1} \left( -\frac{1}{n+1} \right)\right)+\ldots \right).
	\end{align*}
\end{theorem}
\begin{proof}
Let $P$ be an $n$-dimensional polytope of any of the three types mentioned above and let $\phip$ be, respectively, $\phicu$, $\phicr$ or $\phis$.
	
%	By Proposition \ref{prop:general_position_almost_surely} the random linear space $L$ almost surely is in general position with respect to $\cF(P)$.
By Equation~\eqref{eq:random_sections_first_identity_exp}, for an arbitrary $(j+l)$-face $F$ we have
	\begin{align}
		\E \phip (j,n-l,n)
		=
		\# ( \cF_{j+l}(P)) &\cdot  \P (F\cap L \neq \emptyset).
	\end{align}
Applying Propositions \ref{prop:Prob_S_cap_F_notempty_Cube}, \ref{prop:Prob_S_cap_F_notempty_crosspoly} and \ref{prop:Prob_S_cap_F_notempty_simplex} respectively and using the well-known numbers of $(j+l)$-faces of regular polytopes, we directly get
	\begin{align*}
		&\E \phicu (j,k,n) = 2^{k-j} \binom{n}{n-k+j} \cdot 2 \left(
		\upsilon_{n-k+1} (\Ccu_{n-k+j}(k-j)) +\upsilon_{n-k+3}(\Ccu_{n-k+j}(k-j))+\ldots
		\right),
		\\
		&\E \phicr(j,n-l,n)  =2^{j+l+1} \binom{n}{j+l+1} \cdot 2^{-j-l} \left(
		\binom{j+l}{l} + \binom{j+l}{l+1} + \ldots
		\right),\\
		&\E\phis(j,n-l,n)=\\
		%=
		%\# ( \cF_{j+l}(P_{n+1})) \cdot  \P (F\cap S \neq \emptyset) \\
		&\qquad \quad =
		\binom{n+1}{j+l+1} \cdot
		2\left(\upsilon_{l+1} \left(\Cs_{j+l+1} \left( -\frac{1}{n+1} \right)\right)+\upsilon_{l+3}\left(\Cs_{j+l+1} \left( -\frac{1}{n+1} \right)\right)+\ldots \right).
	\end{align*}
\end{proof}

%\begin{remark}
%The above expectations are invariant under dilations  of $P$. In the proofs of Theorem \ref{thm:random_sections_polytopes} and the three propositions before instead of $P$ itself we only %used the isometry types of the cones $C$ spanned by the faces $B$ and the number of $(j+l)$-faces of $P$. Dilations of  the whole centred polytope by a positive factor $\lambda>0$ do not %alter neither these cones nor the numbers of $P$'s $(j+l)$-faces.
%\end{remark}

\subsection{Asymptotic results for the expected number of faces in random sections} \label{subsec_random_section_asymptotic}
%In general, it remains an open question to find the exact asymptotics of the expectations computed above in all possible regimes of the parameters $j,k,n$. So far, only results in some special regimes are available.
Asymptotic behavior of the expected number of faces in sections (or, dually, projections) of regular polytopes  have been studied in the works of Vershik and Sporyshev~\cite{vershik_sporyshev}, B\"or\"oczky and Henk~\cite{boroczky}, Lonke~\cite{lonke} and Donoho and Tanner~\cite{donoho_tanner_ProcAcadUSA,donoho_tanner_JAMS,donoho_tanner_DiscrCompGeom}.
More specifically,  Lonke~\cite[(2) or Corollary 3.4]{lonke} proved the following asymptotic formula for $\phicu(n-m,n-l,n)$ for fixed codimensions $1\leq l<m$:
$$
\E\phicu(n-m,n-l,n) \sim \frac{(2n)^{m-l}}{(m-l)!} \quad \text{ as } n\to\infty.
$$
Here, $a_n\sim b_n$ means that $\lim_{n\to\infty} a_n/b_n = 1$.  B\"or\"oczky and Henk~\cite{boroczky} proved an asymptotic formula for $\phicu(j,k,n)$ for fixed $0\leq j <k$:
%\begin{corollary}\label{cor:asymp_phicu_j_k_n}
%For fixed $0\leq j<k\in \N$ there is a constant $C(j,k)$, such that
$$
\E\phicu(j,k,n) \sim
C(j,k)
\cdot
(\log (n))^{\frac{k-1}{2}}
$$
as $n \to \infty$, where the constant $C(j,k)$ can be expressed by
	$$
		C(j,k)=
		\frac{2^k \pi^{\frac{k-1}{2}} \sqrt{k}  (k-1)!  }{(k-j)! j!} \gs_{j}\left( \frac{1}{k-j} \right).
	$$
%\end{corollary}
Recalling that $\gs_0\equiv 1$, the special case $j=0$ recovers a formula Lonke has proven in~\cite[(5)]{lonke}. Based on Theorem~\ref{thm:random_sections_polytopes} we can complement these results by the following new asymptotic regimes.

\begin{corollary} \label{cor:asymp_phis}
Fix some integers $i > l \geq 1$. As $n\to \infty$, we have
		$$
		\binom{n+1}{n-i+l+1} - \E\phis(n-i,n-l,n) \sim C(i,l) \cdot n^{-\frac{n}{2}} n^{\frac{3i-3}{2}} \left( \frac{(i-l) \eee}{2\pi} \right)^{\frac{n}{2}},
		$$
where
		$$
		C(i,l):= \frac{\pi^{\frac{i-2}{2}} 2^\frac{i-2l+1}{2} \eee^{\frac{3l-3i}{2}}}{(l-1)!(i-l)! (i-l)^\frac{i-1}{2}}.
		$$
\end{corollary}

\begin{corollary} \label{cor:asymp_phicu_j_n-l_n}
Fix some integers $j\geq 0$ and $l\geq 1$. Then, as $n\to \infty$, we have
$$
\E \phicu(j,n-l,n)
\sim
2^{n-j} \frac{n^{j+(l/2)}}{l!j!}  \frac{\Gamma\left( \frac{l+1}{2} \right)}{\pi^{(l+1)/2}}.
%	\binom{l+j}{l} \left(\frac{2}{\sqrt{\pi}} \right)^{l}   \frac{\Gamma\left( \frac{l+1}{2} \right)}{\sqrt{2}} \cdot \frac{1}{n^\frac{l}{2}}.
	$$
\end{corollary}

\section{Proofs}\label{sec:proofs}
\subsection{Cones, angles, conic intrinsic volumes - Proofs} \label{sec:proofs_cones_and_angles}
%\subsection{Proof of Proposition \ref{prop:alt_forms_of_cones}}
\begin{proof}[Proof of Proposition \ref{prop:alt_forms_of_cones}]
We start with the proof of identity \eqref{eq:Ccr_alt_form}.
To prove it we shall show that the cone
$$
C:=\left\{ x = (x_1,\ldots,x_{n+1})\in\R^{n+1} : x_{n+1}\geq \sigma \sum_{i=1}^{n} |x_i| \right\}
$$
is the positive hull of the vectors
$v_i^+ :=\sigma \eee_{n+1}+ \eee_i$ and $v_i^-:=\sigma \eee_{n+1}- \eee_i$, $i=1,\ldots, n$.
%$\eee_{n+1}\pm \eee_1,\ldots, \eee_{n+1}\pm \eee_n$.

Fix any $x=(x_1,\ldots,x_{n+1})\in C$. It can be written in the form
\begin{align*}
x
&=
\sum_{i=1}^n (x_i \eee_i + \sigma|x_i|\eee_{n+1}) + \left(x_{n+1}-\sigma \sum_{i=1}^n |x_i|\right) \eee_{n+1}\\
&=
\sum_{i=1}^n |x_i| (v_i^+ \indi{x_i\geq 0}  + v_i^- \indi{x_i<0}) +\frac{x_{n+1}- \sigma\sum_{i=1}^n |x_i|}{2\sigma} (v_1^++v_1^-),
\end{align*}
where in the second line we used that $v_1^++v_1^-=2\sigma \eee_{n+1}$. Since $x\in C$, we have  $x_{n+1}-\sigma \sum_{i=1}^n |x_i| \geq 0$.
Hence, the coefficients in the above representation are non-negative and we conclude that $x\in \pos(v_i^+, v_i^- : i=1,\ldots,n)= \Ccr_n(\sigma^2)$, thus proving that  $C\subset \Ccr_n(\sigma^2)$. The converse inclusion $\Ccr(\sigma^2) \subset C$
is obvious, since every $v_i^+$ and every $v_i^-$ is in $C$.
%Since for every two elements $x,y\in C$ and any positive factor $\lambda>0$ we have $\lambda x+y\in C$, $C$ obviously is a cone and hence $C=\Ccr_n(\sigma^2)$.

Now we proceed to the proof of \eqref{eq:Ccu_alt_form}.
This time consider the cone
$$
C:=\left\{ x=(x_1,\ldots,x_{n+1})\in\R^{n+1} : x_{n+1}\geq \sigma \max_{1\leq i \leq n }|x_i| \right\}.
$$
Take some $x\in C$. To prove that $x\in \Ccu_n(\sigma^2)$, we shall show that $x$ is contained in the positive hull of the vectors $\sigma \eee_{n+1} + \eps$, where $\eps\in \{-1,1\}^n\subset \R^n$. This is evident if $x_{n+1}=0$ (since then $x=0$). Therefore, let $x_{n+1}>0$. Then, we can write
$
x = \sigma^{-1} x_{n+1}(y_1,\ldots,y_n,y_{n+1}),
$
where $y_{n+1}=\sigma$ and $(y_1,\ldots,y_n) = \sigma x_{n+1}^{-1} (x_1,\ldots,x_n)\in [-1,1]^n$.
Any point $(y_1,\ldots,y_n)$ in the cube $P:=[-1,1]^n$ can be represented as a convex combination of the vertices of the cube, which form the set $\{-1,1\}^n$. It follows that the point $y=(y_1,\ldots, y_n, \sigma)$, which belongs to the shifted cube $P+\sigma \eee_{n+1}\subset \R^{n+1}$, can be represented as a convex  combination of the points of the form $\sigma \eee_{n+1} + \eps$, where $\eps\in \{-1,1\}^n$. Hence, $x$ can be represented as a positive combination of the same points, thus proving that  $C\subset \Ccu_n(\sigma^2)$. The converse inclusion is evident since $\sigma \eee_{n+1} + \eps \in C$ for every $\eps\in \{-1,1\}^n$.
%It is well known that the $n$-dimensional cube  satisfies
%\begin{align*}
%P &=\conv \{-1,1\}^n,\\
%x=(x_1,\ldots,x_n) \in P &\Leftrightarrow \max_{1\leq i \leq n} |x_i| \leq 1.
%\end{align*}
%Thus for $\sigma>0$ the shifted polytope $P+\sigma \eee^{n+1}\subset \R^{n+1}$ satisfies
%\begin{align*}
%P+\sigma \eee_{n+1}&=\conv \{-1,1\}^n+\sigma \eee_{n+1} = \conv \{\sigma \eee_{n+1} +\eps : \eps\in\{-1,1\}^n \} , \\
%y=(y_1,\ldots,y_{n+1})\in P+\sigma \eee_{n+1} &\Leftrightarrow  \left(  \max_{1\leq i \leq n} |y_i| \leq 1 \text{ and } y_{n+1}=\sigma \right).
%\end{align*}
%%where the sum of the two sets in the first equation is the Minkowski sum.
%Calculating the positive hulls in both equations we obtain
%\begin{align*}
%y=(y_1,\ldots,y_{n+1})\in \Ccu_n(\sigma^2)= \pos \{\sigma \eee_{n+1} +\eps : \eps\in\{-1,1\}^n
%\Leftrightarrow
%y\in \pos (P+\sigma \eee_{n+1})\\
%\Leftrightarrow
%\exists \lambda \geq 0 : \left( \max_{1\leq i\leq n} |y_i| \leq \lambda \text{ and }
%y_{n+1}=\sigma \lambda   \right)
%\Leftrightarrow
%\max_{1\leq i\leq n} |y_i| \leq \frac{y_{n+1}}{\sigma},
%\end{align*}
%which proves \eqref{eq:Ccu_alt_form}.
\end{proof}

\begin{proof}[Proof of Proposition \ref{prop:duality_cones}]
Let $F$ be a face of $\Ccu_n(\sigma^2)$ of dimension $n$. Note that $\Ccu_n(\sigma^2)$ has dimension $n+1$.
There is an index $i\in\{1,\ldots,n\}$ and a sign $\tau \in\{-1,1\}$ such that
\begin{align}
F=\pos\bigg\{ \sigma \eee_{n+1}+\tau \eee_i +\sum_{\substack{1\le j\le n\\j\neq i}} \eps_j \eee_j : \eps_1,\ldots,\eps_n \in \{-1,1\}\bigg\} \label{eq:F}.
\end{align}
The linear space spanned by $F$ is
$$
\lin F = \bigg\{\sigma \eee_{n+1}+\tau \eee_i + \sum_{\substack{1\le j\le n\\j\neq i}} \alpha_j \eee_j : \alpha_1,\ldots,\alpha_n \in \R\bigg\}.
$$
The orthogonal complement of the vector space generated by $F$ is therefore
\begin{align*}
(\lin F)^\perp =\Span (-\frac 1\sigma\eee_{n+1} + \tau \eee_i).
\end{align*}
Moreover, $v:=-\frac 1\sigma\eee_{n+1} + \tau \eee_i\in F^\perp$ is a vector in this space that has non-negative scalar products to all vectors in $\Ccu_n(\sigma^2)$.
Since the polar cone $\left(\Ccu_n(\sigma^2)\right)^\circ$ is spanned by the orthogonal complements of the faces of dimension $n$, it is given by
$$
\left(\Ccu_n(\sigma^2)\right)^\circ
=
\pos\left(-\frac{1}{\sigma}\eee_{n+1}\pm \eee_i:i\in\{1,\ldots,n\}\right)
=
-\Ccr_n\left(\frac{1}{\sigma^2}\right).
$$
The second claim of the proposition follows from the first claim (with $\sigma$ replaced by $1/\sigma$) together with $C^{\circ\circ} = C$.
%Denoting the basis vectors $v^\pm_i := -\frac{1}{\sigma}\eee_{n+1}\pm \eee_i$, their scalar products satisfy
%\begin{align*}
%\langle v_i^+,v_i^+ \rangle=\langle v_i^-,v_i^- \rangle=\frac{1}{\sigma^2}+1, \qquad
%\langle v_i^+,v_i^- \rangle=\frac{1}{\sigma^2}-1
%\intertext{and for $i\neq j$}
%\langle v_i^+,v_j^+ \rangle=
%\langle v_i^-,v_j^- \rangle=
%\langle v_i^+,v_j^- \rangle=
%\frac{1}{\sigma^2} .
%\end{align*}
%This proves $\left(\Ccu_n(\sigma^2)\right)^\circ\cong\Ccr_n(\frac{1}{\sigma^2})$ by \eqref{eq:Ccr_scalar_products}.
\end{proof}

\begin{proof}[Proof of Proposition \ref{prop:inner_outer_angle}]
Fix $l \in \{ 1,\ldots,n \}$.
We will determine the tangent cone and the normal cone of an $(l-1)$-dimensional face $F$ of the $n$-dimensional crosspolytope $P_{n} \subset \R^{n}$.
Without loss of generality let $F$ be the face given by
$$
F=\left\{ (f_1,\ldots,f_n)\in\R^n: f_1,\ldots,f_l\geq 0, f_{l+1}=\ldots=f_n=0,\sum_{i=1}^l f_i=1 \right\}.
$$
Fix any point $f\in \relint F$ meaning that $f_1,\ldots,f_l>0$.
The tangent cone $T_F:=T_F(P_{n})$ is the set of all $v\in \R^n$ satisfying
$$
f+\eps v \in P_{n}
$$
for an $\eps >0$. Since the crosspolytope is the unit ball of the $1$-norm on $\R^n$, $P_{n}$ can be characterized via
$$
P_{n}=\left\{u=(u_1,\ldots,u_n)\in \R^n:\sum_{j=1}^n|u_j|\le 1\right\} .
$$
Hence the tangent cone is given by
\begin{align*}
T_F = \{ v\in\R^n: \text{There is } \eps >0 : f+v \eps \in P_{n} \}
&=  \left\{ \sum_{i=1}^l v_i + \sum_{i=l+1}^n|v_i| \le 0 \right\} \\
&= \left\{ -\sum_{i=1}^l v_i \ge \sum_{i=l+1}^n|v_i| \right\} .
\end{align*}
The lineality space of the tangent cone is
$$
\Lineal(T_F)=T_F \cap (-T_F)
= \left\{ \sum_{j=1}^lv_j=0 \right\} \cap \left\{ v_{l+1}=\ldots=v_n=0 \right\} .
$$
The polytope's internal solid angle is the angle of the cone
\begin{align*}
D_{n,l} := T_F \cap (\Lineal(T_F))^\perp = \{ v_1=\ldots=v_l \} \cap \left\{ -l v_1 \ge \sum_{j=l+1}^n|v_j|  \right\}.
%\overset{\text{falsch?}}{\cong} \left\{ x\in \R^{n-l+1} : -l v_{n-l+1} \geq \sum_{j=1}^{n-l}|v_j| \right\} \cong \Ccr_{n-l}(\frac{1}{l^2})
\end{align*}
Lemma \ref{Lem:D_n=C_n-k} below  states that $D_{n,l}$ is isometric to $\Ccr_{n-l}(1/l)$. Together with Corollary \ref{cor:angles_cones} this proves that the internal solid angle has the form \eqref{eq:crosspoly_inner_solid_angle} with $k=l-1$.%Note that dimension of the face $F$ was denoted by $k-1$, whereas in the proposition it is $k$.
%Now we come to the normal cone.
%Let $G$ be a face of dimension $k$.

The normal cone $N_F = N_F(P_n)$ is defined to be the polar cone of $T_F$, hence
$$
N_F \cong \big(D_{n,l} + \Lineal (T_F\big))^\circ
=
D_{n,l}^\circ \cap \big( \Lineal (T_F) \big)^\bot,
$$
where $\cong$ denotes isometry of cones.
In other words, $N_F$ is the polar cone of $D_{n,l}$ with respect to the ambient space $\big(\Lineal (T_F)\big)^\bot$.
Recalling that $D_{n,l}\cong \Ccr_{n-l}\left( 1/l \right)$,  $\Ccr_{n-l}\left( 1/l \right) \subset \R^{n-l+1}$ and $\dim \big(\Lineal (T_F)\big)^\bot = n-l+1$, we have
$$
N_F \cong \left(\Ccr_{n-l}\left( \frac{1}{l} \right) \right)^\circ \cong \Ccu_{n-l}(l),
$$
where we used Proposition \ref{prop:duality_cones} in the last step. Again, Corollary \ref{cor:angles_cones} gives \eqref{eq:crosspoly_outer_solid_angle} with $k=l-1$.
\end{proof}
To complete the proof of Proposition \ref{prop:inner_outer_angle}, it remains to establish the following
\begin{lemma} \label{Lem:D_n=C_n-k}
For $n\in \{2,3,\ldots\}$ and $k\in \{1\ldots,n-1\}$ the cone
$$
D_{n,k}=\left\{ (v_1,\ldots,v_n)\in \R^n: v_1=\ldots=v_k , -k v_1\geq \sum_{j=k+1}^{n}|v_j| \right\}
$$
is isometric to $\Ccr_{n-k}\left( \frac{1}{k} \right)$.
\end{lemma}
\begin{proof}
	We will show that
	\begin{align}
	D_{n,k}= \pos\{u_1^+,\ldots,u_{n-k}^+, u_1^-,\ldots,u_{n-k}^-\} , \label{eq:Lem_inner_solid_angle_goal}
	\end{align}
	where for $i\in \{1,\ldots,n-k\}$
	\begin{align*}
	u_i^+:=-\sum_{j=1}^{k} \eee_j + k \eee_{k+i}, \qquad
	u_i^-:=-\sum_{j=1}^{k} \eee_j - k \eee_{k+i}.
	\end{align*}
	Obviously every $u_i^+$ and every $u_i^-$ is an element of $D_{n,k}$. Hence it is sufficient to show that $D_{n,k} \subset \pos\{u_1^+,\ldots,u_{n-k}^+, u_1^-,\ldots,u_{n-k}^-\}$.
	
	Fix any $v=(v_1,\ldots,v_n)\in D_{n,k}$. By definition
	\begin{align}
	v_1=\ldots=v_k, \label{eq:Lemma_inner_solid_angle_1} \\
	-k v_1 \geq \sum_{j=k+1}^n|v_j|  \label{eq:Lemma_inner_solid_angle_2}.
	\end{align}
	Let $\sgn(r)=\indi{r\geq 0} - \indi{r<0}$ be the sign function.  The vector
	$$
	x:=\sum_{j=k+1}^{n} \frac{|v_j|}{k} u_{j-k}^{\sgn v_j} \in \pos\{u_1^+,\ldots,u_{n-k}^+, u_1^-,\ldots,u_{n-k}^-\}
	$$
	equals $v$ in the last $n-k$ components and the components of $x$ satisfy $x_1=\ldots=x_k$ and $-k x_1 = \sum_{j=k+1}^n|x_j|$.
	By \eqref{eq:Lemma_inner_solid_angle_2}, we have $v_1 \leq x_1$ and thus with the factor
	$$
	\lambda := x_1-v_1=-\frac{1}{k}\sum_{j=k+1}^n|v_j|-v_1 \geq 0
	$$
	we have
	$$
	v=x-\lambda \sum_{j=1}^{k}\eee_j= x+\lambda \frac{u_1^++u_1^-}{2}  \in \pos\{u_1^+,\ldots,u_{n-k}^+, u_1^-,\ldots,u_{n-k}^-\},
	$$
	thus proving \eqref{eq:Lem_inner_solid_angle_goal}. Hence,
	$$
	D_{n,k} = \pos\left\{\frac{u_1^+}{k},\ldots,\frac{u_{n-k}^+}{k}, \frac{u_1^-}{k},\ldots,\frac{u_{n-k}^-}{k}\right\} .
	$$
	The spanning vectors satisfy
	$$
	\left\langle \frac{u_i^+}{k}, \frac{u_j^+}{k} \right\rangle = \left\langle \frac{u_i^-}{k}, \frac{u_j^-}{k} \right\rangle = \frac{1}{k} + \delta_{i,j}, \qquad
	\left\langle \frac{u_i^+}{k}, \frac{u_j^-}{k} \right\rangle = \frac{1}{k} - \delta_{i,j}
	$$
	and thus by \eqref{eq:Ccr_scalar_products} the lemma is proven.
\end{proof}

To prove Theorem \ref{thm:intrinsic_volumes} we will need the following two lemmas.
\begin{lemma} \label{lem:face_Ccu_nomal_cone}
	Fix $n\in \N$ and $k \in \{1,\ldots,n\}$.
	Let $F$ be a $k$-face of the cone $\Ccu_n(\sigma^2)$.
	Then the normal cone $N_F\big(\Ccu_n(\sigma^2)\big)$ of $\Ccu_n(\sigma^2)$ at the face $F$ is isometric to $\Cs_{n-k+1}\left( \frac{1}{\sigma^2} \right)$.
\end{lemma}
\begin{proof}
	%Let $F$ be a face as above. \\
	Since $\sum_{i=k}^{n} \eee_i$ is a point in the relative interior of a $(k-1)$-dimensional face of the $n$-dimensional cube $[-1,1]^n$, it follows that
	$$
	f:=\sigma \eee_{n+1} + \sum_{i=k}^{n} \eee_i
	$$
	is a point in the relative interior of a $k$-dimensional face $F$ of $\Ccu_n(\sigma^2)$.
	By the symmetry of the cone it is sufficient to show that the normal cone of $\Ccu_n(\sigma^2)$ at this face $F$ is isometric to $\Cs_{n-k+1}\left( \frac{1}{\sigma^2} \right)$.
	
First we describe the tangent cone of $\Ccu_n(\sigma^2)$ at $F$. Note that for any $v\in\R^{n+1}$ and any $\delta >0$,
	$$
	f+\delta v = (\sigma +v_{n+1} \delta)\eee_{n+1}+\sum_{i=k}^{n} (1+v_i \delta) \eee_i+\sum_{i=1}^{k-1}v_i\delta \eee_i.
	$$
	Using this identity and the fact that by \eqref{eq:Ccu_alt_form} the cone $\Ccu_n(\sigma^2)$ can be written as
	$$
	\Ccu_n(\sigma^2) = \left\{ \sum_{i=1}^{n+1} \beta_i \eee_i : \max_{1\le i\le n} |\beta_i|\le \frac{\beta_{n+1}}{\sigma} \right\} ,
	$$
	we have that for every $v\in \R^{n+1}$,
	$$
	f+\delta v \in \Ccu_n(\sigma^2) \text{ for sufficiently small } \delta>0 \iff \max_{i\in\{k,\ldots,n\}} v_i \le \frac{v_{n+1}}{\sigma} .
	$$
	Hence the tangent cone of $\Ccu_n(\sigma^2)$ at $F$ is
	\begin{align*}
	T_F(\Ccu_n(\sigma^2))
	&:=
	\{v\in\R^{n+1} : f+\delta v \in \Ccu_n(\sigma^2) \text{ for some } \delta >0\}\\
	&=
	\bigg\{v\in\R^{n+1} : \max_{i\in\{k,\ldots,n\}} v_i \le \frac{v_{n+1}}{\sigma} \bigg\}.
	\end{align*}
	Using that
	$$
	-T_F(\Ccu_n(\sigma^2))
	=
	\bigg\{v\in\R^{n+1} : \min_{i\in\{k,\ldots,n\}} v_i \ge \frac{v_{n+1}}{\sigma} \bigg\}
	$$
	we conclude that the lineality space of the tangent cone is given by
	\begin{align*}
	\Lineal(T_F(\Ccu_n(\sigma^2)))
	&=T_F(\Ccu_n(\sigma^2))\cap \Big(-T_F(\Ccu_n(\sigma^2))\Big) \\
	&=\bigg\{ v\in\R^{n+1}:v_{k}=\ldots=v_n=\frac{v_{n+1}}{\sigma} \bigg\}.
	\end{align*}
	The tangent cone $T_F(\Ccu_n(\sigma^2))$ is the Minkowski sum of its lineality space and the cone $D$ given by
	$$
	D=\left\{ v\in\R^{n+1}:\max_{i\in\{ k,\ldots,n \}} v_i \le 0, v_{n+1}=0\right\}.
	$$
	Since the smallest linear space containing $T_F(\Ccu_n(\sigma^2))$ is the whole $\R^{n+1}$, the normal tangent cone is the polar cone of $T_F(\Ccu_n(\sigma^2))$.
	It is
	\begin{align*}
	\Big(T_F(\Ccu_n(\sigma^2))\Big)^\circ
	&=
	\Big(\Lineal\big(T_F(\Ccu_n(\sigma^2))\big)\Big)^\perp \cap D^\circ \\
	&=
	\left\{ v\in\R^{n+1}:\sigma v_{n+1}=- \sum_{i=k}^n v_i , v_1=\ldots=v_{k-1}=0 \right\} \\
	&\quad \cap
	\left\{ v\in\R^{n+1}: \min_{i\in\{k,\ldots,n\}}v_{i}\ge 0 , v_1=\ldots=v_{k-1}=0 \right\} \\
	&=
	\left\{ v\in \R^{n+1}:v_1=\ldots=v_{k-1}=0, v_{k},\ldots,v_{n}\geq 0, \sigma v_{n+1}=-\sum_{i=k}^{n}v_i \right\} \\
	&=
	\pos \left\{ -\frac{1}{\sigma} \eee_{n+1}+\eee_j : j=k,\ldots,n \right\}\\
	&\cong
	\Cs_{n-k+1} \left( \frac{1}{\sigma^2} \right) .
	\end{align*}
	The isometry in the last step follows from \eqref{eq:Cs_scalar_products}.
	%can be varified by comparing the scalar products between the basis vectors of the cone above with those between the vectors spanning $\Cs_{n-k+1} \left( \frac{1}{\sigma^2} \right)$ by definition.
\end{proof}
\begin{lemma} \label{lem:face_Ccube=Ccube}
	For $n\in \N$, $k\in \{1,\ldots, n+1\}$ and $\sigma^2>0$, any $k$-dimensional face of the cone $\Ccu_n(\sigma^2)$ is isometric to $\Ccu_{k-1}(\sigma^2+n-k+1)$.
\end{lemma}
\begin{proof}
%	Fix $n\in\N$, $\sigma^2>0$ and $k\in \{1,\ldots,n+1\}$.
	Since all the $k$-faces of $\Ccu_n(\sigma^2)$ are isometric, it is sufficient to analyse
	$$
	F:=\pos \left\{ \sigma \eee_{n+1} + \sum_{i=1}^{k-1}\eps_i \eee_i + \sum_{i=k}^{n}\eee_i : \eps_i \in \{-1,1\} \right\}.
	$$
	Denoting the vectors spanning $F$ by
	$$
	v_\eps := \sigma \eee_{n+1} + \sum_{i=1}^{k-1}\eps_i \eee_i + \sum_{i=k}^{n}\eee_i, \qquad \eps=(\eps_1,\ldots,\eps_{k-1})\in \{-1,1\}^{k-1},
	$$
	we have the scalar products
	\begin{align*}
	\langle v_\eps,v_\eta \rangle
	=
	\sigma^2 + \langle \eps,\eta \rangle + n-k+1
	\end{align*}
	for every $\eps,\eta \in \{-1,1\}^{k-1}$.
Since these coincide with the ones of the vectors spanning $\Ccu_{k-1}(\sigma^2+n-k+1)$, see~\eqref{eq:Ccu_scalar_products}, the claimed isometry holds.
\end{proof}

\begin{proof}[Proof of Theorem \ref{thm:intrinsic_volumes}]
	First let $k\in \{ 1,\ldots,n\}$.
	Let $F$ be a $k$-face of $\Ccu_n(\sigma^2)$.
	By Lemmas \ref{lem:face_Ccu_nomal_cone} and \ref{lem:face_Ccube=Ccube} we have
	\begin{align*}
	\alpha(N_F(\Ccu_n(\sigma^2)))=\gs_{n-k+1}\left(\frac{1}{\sigma^2}\right), \\
	\alpha(F)=\gcu_{k-1}(\sigma^2+n-k+1) .
	\end{align*}
	By definition, the $k$\textsuperscript{th} intrinsic volume of $\Ccu_n(\sigma^2)$ is
	$$
	\upsilon_k(\Ccu_n(\sigma^2))=\sum_{F\in\cF_k(\Ccu_n(\sigma^2))} \alpha(F) \alpha(N_F(\Ccu_n(\sigma^2))) .
	$$
	All the $k$-faces of $\Ccu_n(\sigma^2)$ are isometric and by the construction of the cone there is a natural one-to-one-correspondence between the $k$-faces of $\Ccu_n(\sigma^2)$ and the $(k-1)$-faces of the $n$-dimensional cube $[-1,1]^n$. Hence the cone $\Ccu_n(\sigma^2)$ has $2^{n-k+1} \binom{n}{k-1}$ $k$-faces and its intrinsic volume is
	$$
	\upsilon_k\left( \Ccu_n(\sigma^2) \right)
	=
	2^{n-k+1} \binom{n}{k-1} \gcu_{k-1}\left(\sigma^2+n-k+1\right) \gs_{n-k+1}\left( \frac{1}{\sigma^2} \right) .
	$$
	Now we are coming to the remaining cases.
%The only $0$-dimensional face of $\Ccu_n(\sigma^2)$ is the origin and its only $n+1$-dimensional face is the apex $\{0\}$.
Since $\upsilon_{n+1}(\Ccu_n(\sigma^2))$ is just the solid angle of $\Ccu_n(\sigma^2)$, we have
$$
\upsilon_{n+1}(\Ccu_n(\sigma^2))
%=
%\alpha(\{\Ccu_n(\sigma^2))\}) \cdot \alpha(N_{\Ccu_n(\sigma^2))}(\Ccu_n(\sigma^2)))
%=
%\left(\Ccu_n(\sigma^2)\right) \cdot 1
=
\gcu_n(\sigma^2).
	$$
Recalling that $\gs_0(\cdot) \equiv 1$, this gives that \eqref{eq:Ccu_intrinsic_volume} also holds for $k=n+1$. In the case  $k=0$ we observe that the only $0$-dimensional face of $\Ccu_n(\sigma^2)$ is $\{0\}$ and hence,  by Proposition \ref{prop:duality_cones},
	$$
	\upsilon_0(\Ccu_n(\sigma^2))
	=
	\alpha(\{0\}) \cdot \alpha(N_{\{0\}}(\Ccu_n(\sigma^2)))
	=
	1\cdot \alpha\left((\Ccu_n(\sigma^2))^\circ\right)
	=
	\gcr_n\left(\frac{1}{\sigma^2}\right).
	$$

The conic intrinsic volumes of $\Ccr_n(\sigma^2)$ can be obtained by polarity. It is known, see~\cite[Fact 5.5(2)]{amelunxen}, that for every $m$-dimensional cone $C$ the $j$\textsuperscript{th} intrinsic volume of its polar cone $C^\circ$ is
	\begin{align}
	\upsilon_j(C^\circ)=\upsilon_{m-j}(C), \qquad j=0,\ldots,m.
	\end{align}
	Since by Theorem \ref{prop:duality_cones} we have that $\Ccr_n(\sigma^2)\cong \left( \Ccu_n \left(  \frac{1}{\sigma^2} \right) \right)^\circ$, it follows that
	%for every $n\in\N$ and every $1\leq k\leq n$
	$$
	\upsilon_k\left( \Ccr_n(\sigma^2) \right)
	=
	\upsilon_{n+1-k}\left( \Ccu_n \left( \frac{1}{\sigma^2} \right) \right)
	=
	2^{k} \binom{n}{k} \gcu_{n-k}\left(\frac{1}{\sigma^2}+k\right) \gs_{k}\left( \sigma^2 \right),
	$$
	for all $0\leq k \leq n$. In the remaining case $k=n+1$ we have
	$$
	\upsilon_{n+1}(\Ccr_n(\sigma^2))
	=
	\upsilon_{0}\left(\Ccu_n \left( \frac{1}{\sigma^2} \right) \right)= \gcr_n(\sigma^2),
	$$
which completes the proof.
\end{proof}

\subsection{Absorption probabilities - Proofs}
In this section we will express the non-absorption probabilities as the probability that certain  deterministic cone $C$ and a random linear space $L$ intersect trivially.
The latter probability can be computed by means of the conic Crofton formula which is an important tool in this and the next section.
	\begin{theorem}[Conic Crofton formula] \label{thm:conic_crofton}
		Let $C\subset \R^n$ be a convex cone which is not a linear subspace and let $L$ be a random linear subspace that is chosen uniformly from the Grassmannian $Gr(n-l, \R^n)$, i.e.\ the set of all linear subspaces of $\R^n$ that have codimension $l\in \{0,\ldots,n\}$. Then
		\begin{align*}
		\P (C\cap L=\{0\})=2(\upsilon_{l-1}(C)+\upsilon_{l-3}(C)+\ldots),\\
		\P (C\cap L\neq \{0\})=2(\upsilon_{l+1}(C)+\upsilon_{l+3}(C)+\ldots).
		\end{align*}
	\end{theorem}
	The probabilities $\P (C\cap L\neq \{0\})$ are known as the Grassmann angle $\gamma_l(C)$.
	\begin{proof}[Proof of Theorem \ref{thm:p_symm}]
		%Fix $n,d \in \N$ and $\sigma>0$ and let $X$ be a $d$-dimensional standard normal distributed random variable, that is independent of the variables $X_1,\ldots,X_n$ defining $\Psym_n$.
By the symmetry of the standard normal distribution we have
\begin{equation*}
		\psym_{n,d}(\sigma^2)
		=
		\P(\sigma X\notin \Psym_{n,d}) = \P(-\sigma X\notin \Psym_{n,d})
		=
		\P(0\notin \conv \{\pm X_1+\sigma X,\ldots,\pm X_n+\sigma X\}).
\end{equation*}
 By the definition of convex hulls, the event $0\notin \conv \{ \pm X_1+\sigma X,\ldots,\pm X_n+\sigma X \}$ occurs if and only if
$$
0=\sum_{i=1}^n \alpha_i X_i+\alpha_{n+1}  X
$$
for $\alpha_1,\ldots,\alpha_n,\alpha_{n+1}\in\R$ and $\alpha_{n+1} \geq \sigma \sum_{i=1}^n|\alpha_i|$
implies that $\alpha_1=\ldots =\alpha_n=\alpha_{n+1} = 0$.
Taking everything together, we arrive at the identity
\begin{align}
\psym_{n,d}(\sigma^2)
=
\P(C\cap U= \{0\} )\label{eq:C_cap_U=0}
\end{align}
		with
		\begin{align}
		U:&= \left\{ (y_1,\ldots,y_{n+1}) \in \R^{n+1} : \sum_{i=1}^n y_i X_i+y_{n+1} X=0 \right\}, \\
		C:&= \left\{ (\alpha_1,\ldots,\alpha_{n+1})\in\R^{n+1} : \alpha_{n+1} \ \geq \sigma \sum_{i=1}^n |\alpha_i| \right\} . \label{eq:def_C_symm}
		\end{align}
Note that $U$ is a random linear subspace of $\R^{n+1}$ that has codimension $d$ a.s. Moreover, from the rotational symmetry of the standard normal distribution it follows that $U$ is	 	uniformly distributed on the corresponding linear Grassmannian.
		By \eqref{eq:Ccr_alt_form}, we have $C\cong \Ccr_n(\sigma^2)$ and hence by the conic Crofton formula (Theorem \ref{thm:conic_crofton}) the probability \eqref{eq:C_cap_U=0} takes the form
		$$
		\psym_{n,d}(\sigma^2)=2 ( \upsilon_{d-1} (C)+\upsilon_{d-3}(C)+\ldots )
		=2 ( \upsilon_{d-1} (\Ccr_n(\sigma^2))+\upsilon_{d-3}(\Ccr_n(\sigma^2))+\ldots ).
		$$
Plugging in the expressions from Theorem \ref{thm:intrinsic_volumes} completes the proof.
	\end{proof}

\begin{proof}[Proof of Theorem \ref{thm:p_zono}]
	As in the proof of Theorem \ref{thm:p_symm}  we have
	$$
	\pzono_{n,d}(\sigma^2)= \P(\sigma X \notin \Pzono_{n,d}) = \P(0 \notin \Pzono_{n,d}+\sigma X).
	$$
	Noting that $$\Pzono_{n,d}+\sigma X= \left\{ \sum_{i=1}^{n}\tilde \alpha_i X_i + \sigma X : \tilde \alpha_1,\ldots,\tilde \alpha_n \in [-1,1] \right\}$$ we can rewrite the event $\{0\notin \Pzono_{n,d}+\sigma X\}$ in the following way:
	\begin{align*}
	0 \notin P_{n,d}+\sigma X
	&\Leftrightarrow
	\text{ for all } \tilde \alpha_1,\ldots,\tilde \alpha_n\in [-1,1]: \sum_{i=1}^{n} \tilde \alpha_i X_i + \sigma X \neq 0 \\
	&\Leftrightarrow
	\text{ for all } \lambda>0, \tilde \alpha_1,\ldots,\tilde \alpha_n\in [-1,1] : \sum_{i=1}^{n} \lambda \tilde \alpha_i X_i + \lambda \sigma X \neq 0 .
	\end{align*}
	Denoting $\alpha_i:=\lambda \tilde \alpha_i$ and  $\alpha_{n+1} := \lambda \sigma$ and including the case $\alpha_{n+1}=0$ this event is equivalent to the statement
	$$
	\left( \sum_{i=1}^{n} \alpha_i X_i + \alpha_{n+1} X = 0 \text{ with } \alpha_{n+1}\geq 0 \text{ and } \max_{i=1,\ldots,n} |\alpha_i| \leq \frac{\alpha_{n+1}}{\sigma} \right) \Rightarrow \alpha_1=\ldots=\alpha_{n+1}=0.
	$$
	Thus $\pzono_{n,d}(\sigma^2)$ has the form
	\begin{align}
	\pzono_{n,d}(\sigma^2)=\P(C\cap U=\{0\})
	\end{align}
	with
	\begin{align}
	U:&= \left\{ (y_1,\ldots,y_{n+1}) \in \R^{n+1} : \sum_{i=1}^n y_i X_i+y_{n+1} X=0 \right\}, \\
	C:&= \left\{ (\alpha_1,\ldots,\alpha_{n+1})\in\R^{n+1} : \alpha_{n+1}  \geq \sigma \max_{i=1,\ldots,n} |\alpha_i| \right\} . \label{eq:def_zono}
	\end{align}
Note that $U$ is the same as in the proof of Theorem \ref{thm:p_symm}. Thus, $U$ is a random linear subspace of $\R^{n+1}$ that has codimension $d$ and is uniformly distributed on the corresponding linear Grassmannian.
Also, $C \cong \Ccu_n(\sigma^2)$ by \eqref{eq:Ccu_alt_form} and hence by the conic Crofton formula this probability equals
	$$
	\pzono_{n,d}(\sigma^2)= 2 (\upsilon_{d-1}(\Ccu_n(\sigma^2))+\upsilon_{d-3}(\Ccu_n(\sigma^2))+\ldots).
	$$
	Plugging in the expressions from Theorem \ref{thm:intrinsic_volumes} completes the proof.
\end{proof}

%\subsection{Random Sections - Proofs}
\subsection{Proof of Proposition \ref{Prop:Random_sections_first_subsection}}\label{sec:proof_intersection}
The main concept we need in the proof of \eqref{eq:random_sections_first_identity} is the general position of a linear subspace with respect to a finite set of affine subspaces. This notion is defined in the following way.

Let $S_1, \ldots, S_k\subset \R^n$ be affine subspaces.
%with dimensions $\dim{S_i}=s_i$, $i=1,\ldots,k$.
An affine subspace $S\subset \R^n$ having codimension $\codim S = l$ is said to be in general position with respect to $S_1,\ldots,S_k$ if for every choice of indices $I\subset \{1,\ldots, k\}$ the intersection $S\cap \left( \bigcap_{i\in I} S_i \right)$ either is empty if $\dim \left( \bigcap_{i\in I} S_i \right)<l$ or is an affine subspace of dimension $\dim \left( \bigcap_{i\in I} S_i \right)-l$ otherwise.

The next proposition implies that for any finite collection of affine subspaces $S_1,\ldots,S_k$ that are not linear subspaces, with probability $1$ a random linear subspace $L$ which is chosen uniformly from the Grassmannian of all linear subspaces with codimension $l$ is in general position with respect to $S_1,\ldots,S_k$. It is sufficient to prove this fact for a single affine subspace $A$ because every intersection $\bigcap_{i\in I} S_i$ is an affine subspace.
% itself and thus $L$ almost surely is in general position to it.
%The following proposition gives the proof.
\begin{proposition} \label{prop:general_position_almost_surely}
Fix $n\in \N$ and $l\in \{0,\ldots,n\}$. 	Let $L$ be a random linear subspace of $\R^n$ chosen uniformly from the set of all $(n-l)$-dimensional linear subspaces of $\R^n$ and let $A\subset \R^n$ be an affine linear subspace of $\R^n$ that is not a linear subspace.
	Then $L$ almost surely is in general position with respect to $A$.
	\end{proposition}
	\begin{proof}
		The way $L$ is constructed it almost surely is in general position with respect to any fixed deterministic \textit{linear} subspace $M$. Proofs for this statement are given by Goodey and Schneider in \cite[Lemma 2.1]{goodey_schneider} and by Schneider and Weil in \cite[Lemma 13.2.1]{SW08}.
		
		Thus $L$ almost surely is in general position with respect to $\lin(A)$, the linear hull of $A$, i.e.\ almost surely the dimension of $\lin(A)\cap L$ is either $0$, if $\dim (\lin (A))\leq  l$, or it is $\dim(\lin(A))-l$ otherwise. Note that this definition of general position of linear subspaces slightly differs from the one of general position of a linear subspace with respect to an affine subspace.
		
		Unfortunately, general position of $L$ with respect to $\lin(A)$ is not equivalent to general position of $L$ with respect to $A$, as one can see by analysing two parallel lines in $\R^3$ with one of them containing $0$ and thus being a linear subspace.
		But to prove the proposition it is sufficient to show only one of the implications:
		we will prove that the general position of $L$ with respect to $\lin(A)$ implies the general position of $L$ with respect to $A$.
		
		First assume that $L$ is in general position with respect to $\lin(A)$ and $\dim(A)< \codim(L)=l$. Since $A$ is not a linear subspace, $0\notin A$ and thus $\dim(\lin(A))=\dim(A)+1$.
		%$\dim(\lin(A))\leq \codim(L)$
		Hence we have $\dim (\lin (A))\leq  \codim (L)$ and by definition we also have $\dim(\lin(A)\cap L)=0$ and thus $\lin(A)\cap L=\{0\}$.
		Thereby we obtain $A\cap L \subset \lin(A)\cap L=\{0\}$. Since $A$ is an affine subspace, $0\notin A$ and thus $0 \notin A\cap L$, we arrive at $A\cap L =\emptyset$.
		By definition this means that in this case $L$ is in general position with respect to $A$.
		
		Now assume that $\dim(A) \geq \codim(L)=l$, while $L$ is in general position with respect to $\lin(A)$.
		Since $L$ is in general position with respect to $\lin(A)$, the dimension of their intersection is $\dim(\lin(A)\cap L)=\dim(\lin(A))-\codim(L)$ and thus
		\begin{align*}
		\dim(A\cap L)
		=
		\dim(\lin(A\cap L))-1
		=
		\dim(\lin(A)\cap L)-1
		&=
		\dim(\lin(A))-\codim(L)-1\\
		&=
		\dim(A)-\codim(L),
		\end{align*}
		which gives general position.
		In the first step we used that $A\cap L$ is an affine subspace that is not a linear subspace. The identity $\lin(A)\cap L=\lin(A\cap L)$ that we used in the second step, can be proved the following way.
		
		Since $A$ is an affine subspace that is not a linear subspace, we can write $\lin(A)= \bigcup \{A\lambda:\lambda \in \R\}$.
		Thus for any $x\in \lin(A)\cap L$ with $x\neq 0$ there is an $a\in A$ and a real $\lambda\neq 0$ such that $x=\lambda a$. Hence we have $\frac{x}{\lambda}\in A$ and since $L$ is a linear subspace $\frac{x}{\lambda}\in L$. As a result we have $\frac{x}{\lambda}\in A\cap L$ and thus $x\in \lin(A\cap L)$. In the formally excluded case $x=0$ we trivially have $x\in \lin(A\cap L)$. The other inclusion holds trivially: $\lin(A\cap L)$ is the smallest linear subspace containing $A\cap L$ and since $A\cap L\subset \lin(A)\cap L$ we conclude that $\lin(A\cap L)\subset \lin(A)\cap L$.
		%As before $\dim(\lin(A))=\dim(A)+1$ and thus $\dim(\lin(A)) > \codim(L)$.
		%Since $L$ is in general position with respect to $\lin(A)$, the dimension of their intersection is $\dim(\lin(A)\cap L)=\dim(\lin(A))-\codim(L)=\dim(A)-\codim(L)+1$.
	\end{proof}

To prove \eqref{eq:random_sections_first_identity} we will show the following three statements.
\begin{proposition} \label{prop:random_sections_first_identity}
Let $P\subset \R^n$ be an $n$-dimensional polytope that contains the origin in its interior, and let $S \subset R^n$ be a deterministic linear subspace having codimension $l\in \{1,\ldots,n-1\}$ which is in general position with respect to $\{\aff(F):F\in\cF(P)\}$. Here, $\cF(P)$ is the set of all faces of $P$.
Then the following three statements hold for every $j\in \{0, \ldots , n-l-1\}$.
		\begin{enumerate}
			\item \label{item:random_sections_first_identity_1}
			Let $B_1, B_2 \in \cF_{j+l}(P)$ be two $(j+l)$-faces of $P$ with $B_1\cap S \neq \emptyset$ and $B_1\neq B_2$.  Then
			$$
			B_1\cap S \neq B_2\cap S.
			$$
			%In other words  \ldots
			\item  \label{item:random_sections_first_identity_2}
			Let $B\in \cF_{j+l}(P)$ be a $(j+l)$-face of P with $B\cap S \neq \emptyset$. Then the intersection of $B$ and $S$ is a $j$-face of $S\cap P$, i.e.
			$$
			B\cap S \in \cF_j(S\cap P).
			$$
			\item \label{item:random_sections_first_identity_3}
			Let $A \in \cF_j(S\cap P)$ be a $j$-face of the intersection of $S$ and $P$. Then there is a $(j+l)$-face $B\in \cF_{j+l}(P)$ of $P$ such that
			\begin{align} \label{eq:in_item:random_sections_first_identity_3}
				B \cap S=A.
			\end{align}
		\end{enumerate}
	\end{proposition}

\begin{proof}[Proof of Proposition~\ref{Prop:Random_sections_first_subsection} assuming Proposition~\ref{prop:random_sections_first_identity}]
In our setting of a random linear subspace $S=L$ and a deterministic polytope $P$, the subspace $L$ almost surely is in general position with respect to $\cF(P)$  by Proposition \ref{prop:general_position_almost_surely}. It follows that statements~\eqref{item:random_sections_first_identity_1} to \eqref{item:random_sections_first_identity_3} of Proposition~\ref{prop:random_sections_first_identity} hold almost surely.
				
Statement \eqref{item:random_sections_first_identity_3} gives a map that sends $A \in \cF_j(L\cap P)$ to the face $B\in \cF_{j+l}(P)$ satisfying~\eqref{eq:in_item:random_sections_first_identity_3}. By statement~\eqref{item:random_sections_first_identity_1} this $B$ is unique and thus the map is well-defined.
Since by~\eqref{eq:in_item:random_sections_first_identity_3} the map is injective, we have
$$
\phip(j,n-l,n) = \#\cF_j(L\cap P)  \leq \# \{ B \in \cF_{j+l}(P) : B\cap L \neq \emptyset \} = \sum_{B\in \cF_{j+l}(P)} \indi{B \cap L \neq \emptyset} .
$$
By statement \eqref{item:random_sections_first_identity_2} the image of the map
%for every $B \in \cF_{j+l}(P)$ there is a face $A := B\cap L \in \cF_j(L\cap P)$ and by statement \eqref{item:random_sections_first_identity_1} this map
$B \mapsto B\cap L$ from $\{B \in \cF_{j+l}(P) : B\cap L \neq \emptyset \}$ to $\cF_j(L\cap P)$ is a subset of $\cF_j(L\cap P)$, and by statement \eqref{item:random_sections_first_identity_1} this map in injective. Thus
$$
\phip(j,n-l,n)  = \#\cF_j(L\cap P) \geq \# \{ B \in \cF_{j+l}(P) : B\cap L \neq \emptyset \} = \sum_{B\in \cF_{j+l}(P)} \indi{B \cap L \neq \emptyset}.
$$
So, Proposition~\ref{Prop:Random_sections_first_subsection} is a  consequence of Proposition \ref{prop:random_sections_first_identity}.
\end{proof}			
				
				%An important concept that we need in the proof of Proposition \ref{prop:random_sections_first_identity} is the general position of affine linear subspaces.

The following lemma is an important step in the proof of Proposition \ref{prop:random_sections_first_identity}. Its proof is inspired by and similar to the proof of \cite[Lemma 3.5]{kabluchko_convex_hulls_hyperplane_arragenments_weyl_chambers}.
	\begin{lemma} \label{lem:random_sections_S_cap_intQ_not_empty}
		Let $Q\subset \R^n$ be a polytope (or, more generally, an intersection of finitely many half-spaces which is allowed to be unbounded) of full dimension $\dim Q=n$.
		Let the linear subspace $S\subset \R^n$ of codimension $l\in \{1,\ldots,n-1\}$ be in general position with respect to the set of the affine hulls of its faces $\{ \aff (F): F\in \cF(Q) \}$.
		If $S$ intersects $Q$, then it also intersects its interior $\interior (Q)$, i.e.
		$$
		S\cap Q \neq \emptyset \Rightarrow S\cap \interior(Q)\neq \emptyset .
		$$
	\end{lemma}
\begin{proof}
Assume that $S\cap Q \neq \emptyset$ but $S\cap \interior(Q) = \emptyset$. It is known that the polytope $Q$ is the disjoint union of the relative interiors of all its faces. Because of $S\cap Q \neq \emptyset$ there is a face $F$ of $Q$ with $\relint(F)\cap S \neq \emptyset$ and by $S\cap \interior(Q) = \emptyset$ we have $F \neq Q$.
							
Without loss of generality let the origin $0\in \relint(F)\cap S$. Then the affine hull of $F$ is its linear hull $\aff(F)=\lin(F)$.
By assumption of the lemma, $S$ is in general position with respect to $\lin(F)$. Thus $S\cap \lin(F)$ is the empty set if $\dim(\lin (F))<l$ or it is a linear subspace of dimension $\dim(F)-l$ else.
Since by assumption $\relint(F)\cap S \neq \emptyset$, the first case is impossible, thus we must have $\dim(F)\geq l$ and $\dim(\lin(F)\cap S)=\dim(F)-l$.
Since $S$ and $\lin(F)$ are linear subspaces, this even implies
	$$
		\dim(S + \lin(F))= \dim(S)+\dim(\lin(F))-\dim(S\cap \lin(F))=n-l+\dim(F)-(\dim(F)-l)=n
	$$
	and thus $S+\lin(F)=\R^n$. Denoting $V_0=\lin(F)\cap S$ this implies the existence of two linear  subspaces $V_1, V_2\subset \R^n$ satisfying $V_0 \bot V_1$, $V_0 \bot V_2$, $\lin(F)=V_0+V_1$, $S=V_0+V_2$ and $V_0+V_1+V_2=\R^n$.
We will show that this implies a contradiction.
							
	Let $T_F(Q)=\{y\in \R^n:\exists \eps>0 \text{ such that } \eps y \in Q\}$ be the tangent cone of $Q$ at $F$. Fix any $z\in \interior(T_F(Q))$. By our result above there is a decomposition $z=v_0+v_1+v_2$ with $v_i\in V_i$, $i=0,1,2$. We will prove that there is an $\eps >0$ such that $\eps v_2\in \interior(Q)$, which is a contradiction to $\eps v_2\in V_2 \subset S$ and $S\cap \interior(Q)=\emptyset$.
							
	Since $v_0 + v_1 \in \lin(F) \subset T_F(Q)$, we have $-(v_0+v_1)\in \lin(F)\subset T_F(Q)$ and thus
	$$
		v_2=z-v_0-v_1 \in T_F(Q).
	$$
	Note that $v_2$ is the projection of $z$ onto $V_2$ along $V_0+V_1$.
	The above argument showing that $v_2\in T_F(Q)$ applies to every point in a sufficiently small ball around $z$. The projection of this ball onto $V_2$ along $V_0+V_1$ covers some set $B_{r'}(v_2)\cap V_2$, where $B_{r'}(v_2)$ is a ball of radius $r'>0$ around $v_2$. Thus $B_{r'}(v_2)\cap V_2 \subset T_F(Q)$. Since $V_0+V_1=\lin(F)\subset T_F(Q)$, by the convex cone property of $T_F(Q)$ there  is an $r\in(0,r']$ such that $B_r(v_2) \subset T_F(Q)$.
	
	Again by the convex cone property of $T_F(Q)$, for every $\eps>0$ we have $B_{\eps r}(\eps v_2) \subset T_F(Q)$. By the definition of the tangent cone $T_F(Q)$, for sufficiently small $\eps$ we have $B_{\eps r}(\eps v_2) \subset Q$ and thus $\eps v_2 \in \interior(Q)$, which is a contradiction as explained above.
\end{proof}
In the proof of statement \eqref{item:random_sections_first_identity_3} in Proposition \ref{prop:random_sections_first_identity} we will need the following corollary from the Hyperplane Separation Theorem. To state it, we need the following definition.
Let $H\subset \R^n$ be an affine hyperplane given by the equation $H=\{x\in \R^n : \langle x,v \rangle=r \}$, $v\in \R^n\backslash\{0\}$, $r\in\R$. Then we define the two half spaces that $H$ divides $\R^n$ into by $H^+:=\{x\in \R^n : \langle x,v \rangle\geq r \}$, $H^-:=\{x\in \R^n : \langle x,v \rangle \leq  r \}$. Note that these spaces swap positions, if $v$ is replaced by $-v$ and $r$ by $-r$. Thus these spaces are not well-defined, if only $H$, but not the exact form of its defining equation is given. So when we speak of $H^+$ it can be any of the two half spaces, but it will always contain $H$.
\begin{lemma} \label{lem:separation_thm}
Let $Q \subset \R^n$ be a convex set with non-empty interior and let $H_0\subset \R^n$ be an affine subspace with $H_0 \cap Q \neq \emptyset$, but $H_0 \cap \interior(Q) = \emptyset$. Then there is an affine hyperplane $H\subset \R^n$ with $H_0 \subset H$ and $Q \subset H^+$.
\end{lemma}
\begin{proof}
In the situation of the Lemma $\interior(Q)$ and $H_0$ are two disjoint convex sets. By the Hyperplane Separation Theorem, see Theorem~1.3.7 in~\cite{schneider_convex_bodies}, there is a hyperplane $H$ such that $H_0 \subset H^-$ and $\interior(Q)\subset H^+$. Since $H_0$ is a linear subspace of $\R^n$ which is contained in the half space $H^-$ it must be parallel to $H$.
By $H_0\cap Q \neq \emptyset$
% implies $\dist(H_0, \interior(Q))=\dist(H_0, Q)=0$ and thus
there is a point $x\in Q\cap H_0$ and by the construction of $H$ we have $x\in H$. Thus $H$ and $H_0$ both contain the point $x$ and hence we have $H_0\subset H$.
Recalling that $\interior(Q)\subset \interior(H^+)$ and taking the closure we conclude that $Q\subset H^+$.
\end{proof}
Now we can prove Proposition~\ref{prop:random_sections_first_identity}.

\begin{proof}[Proof of Proposition \ref{prop:random_sections_first_identity}]
	We first prove statement \ref{item:random_sections_first_identity_1}.
	
	Let $B_1, B_2, S$ be defined as in the statement. We can think of $B_1$ as a polytope in its affine hull $\aff(B_1)$. In this setting $B_1\subset \aff (B_1)$ is a polytope of full dimension and $S\cap \aff(B_1)$ is an affine linear subspace of $\aff(B_1)$ which is in general position with respect to $\{\aff(F): F\in \cF(B_1)\}$.
	
	Lemma \ref{lem:random_sections_S_cap_intQ_not_empty} gives that
	\begin{align}
		(\relint B_1)\cap S
		=
		(\relint B_1)\cap (S\cap \aff(B_1))
		\neq
		\emptyset. \label{eq:S_cap_relint_B1_notempty}
	\end{align}
	Since the dimensions of $B_1$ and $B_2$ are equal, $\relint (B_1)$ and $B_2$ are disjoint.
	By \eqref{eq:S_cap_relint_B1_notempty} there is an $x\in (\relint B_1)\cap S \subset B_1\cap S$ and since $x\in \relint(B_1)$, we have $x\notin B_2 \supset B_2 \cap S$. Thus we have $B_1 \cap S \neq B_2 \cap S$.
	
	To prove statement \eqref{item:random_sections_first_identity_2} fix $B \in \cF_{j+l}(P)$. As a first step we prove that $S\cap B$ is a face of $S\cap P$.
	Since $B$ is a face of $P$, there is an affine hyperplane $H\subset \R^n$ such that $H\cap P=B$ and $P\subset H^+$. Thus we have
	\begin{align*}
		B\cap S &= (H\cap P)\cap S =  H\cap (P\cap S) \;\;\text { and} \\
		P\cap S &\subset P \subset H^+,
	\end{align*}
	which implies by definition that $S\cap B$ is a face of $S\cap P$.
	
	To complete the proof of statement \eqref{item:random_sections_first_identity_2} we show that the dimension of $S\cap B$ is $j$.
	First note that since $S$ is in general position with respect to $\aff(B)$, we have $\dim(\aff(B) \cap S)=j$ and thus $\dim(B\cap S) \leq j$.
	
	By the same argument as in the beginning of the proof of statement \eqref{item:random_sections_first_identity_1} we have $\relint(B)\cap S \neq \emptyset$.
	Thus there is an $x\in \relint(B) \cap S$ and since $\relint (B)$ is a relatively open subset of $\aff(B)$, there is an $\eps>0$ such that $B_\eps(x) \cap \aff(B) \subset \relint(B)$. Here $B_\eps(x)$ denotes the $n$-dimensional ball with radius $\eps>0$ and centre $x$.
	Thus we have
	$$
		B_\eps(x) \cap \aff(B)\cap S \subset \relint(B)\cap S \subset B\cap S.
	$$
	Since $S$ is in general position with respect to $\aff(B)$, the intersection $S\cap \aff(B)$ has dimension $j$ and since $B_\eps(x)$ is an $n$-dimensional ball centred at $x\in S\cap \aff(B)$, intersecting with it does not change the dimension. Thus we have
	$$
		j=\dim(B_\eps(x) \cap \aff(B)\cap S)\leq  \dim(B\cap S).
	$$

	To prove statement \eqref{item:random_sections_first_identity_3}, fix $A \in \cF_j(S\cap P)$.
	We first show that there is a face $B$ of $P$ such that $A=B\cap S$.
	To do this we will construct an affine hyperplane $H\subset \R^n$ such that $H\cap S \cap P = A$ and $P\subset H^+$. In this setting $B:= H\cap P$ is a face of $P$ by definition and $B$ suffices $B\cap S= A$.
	
	Since $A$ is a face of $S\cap P$, there is a hyperplane $H_0 \subset S$ such that $H_0 \cap (S\cap P)= A$ and $S\cap P \subset H_0^+$. Note that $H_0$ is a hyperplane in $S$, but not in $\R^n$.
	
	By Lemma \ref{lem:separation_thm} there is a hyperplane $H \subset \R^n$ such that $H_0 \subset H$ and $P \subset H^+$.
	
	If $H\cap S=H_0$,  the hyperplane $H$ obviously suffices
	$$
	H\cap P \cap S = (H \cap S) \cap (S\cap P) = H_0 \cap (S\cap P)=A.
	$$
	To prove the identity $H\cap S=H_0$, first note that $H_0 = H_0 \cap S \subset H\cap S$. To prove the equality, we assume by contraposition that there exists an $x \in (H\cap S)\setminus H_0$.
	Then $H\cap S$ is an affine linear subspace of $S$ that contains the hyperplane $H_0$ and the point $x \notin H_0$. Thus  $H\cap S = S$ and, since $S$ is a linear subspace, $0 \in S =S\cap H \subset H$. On the other hand, the point $0$ belongs to the interior of $P$ by the assumption of the proposition. This is a contradiction to $P \subset H^+$.
	
	%The last step remaining is to prove that for any face $A \in \cF_j(S\cap P)$ there is a face $\hat B \in \cF_{j+l}(S\cap P)$ satisfying $\hat B \cap S = A$. In other words we have to prove that for every face $A$ of $S\cap P$ there is a face $B$ of $P$, that satisfies $B\cap S =A$ and even has the right dimension. (unnötig kompliziert)
	
It remains to show that the face $B:=H\cap P$ has the right dimension $j+l$.
Recall that  $A\in \cF_j(S\cap P)$. By the first part of the proof of statement \eqref{item:random_sections_first_identity_3} there is $k\geq j$ and a face $B\in\cF_k(P)$ such that $B\cap S=A$.
	
We first assume that the dimension $k$ is smaller than the codimension $l$ of $S$. Since $S$ is in general position with respect to $\aff(B)$, the intersection $S\cap \aff(B)$ is the empty set, which is a contradiction to $A=B\cap S \subset S\cap \aff(B)$. Thus we have $k\geq l$.
	
In this case we have $B \in \cF_{l+(k-l)}(P)$ with $k-l\geq 0$. Hence statement \eqref{item:random_sections_first_identity_2} immediately gives
$$
k-l=\dim(B\cap S)=\dim(A)=j
$$
and thus $k=j+l$.
This completes the proof of statement \eqref{item:random_sections_first_identity_3} and thus of the proposition.																
									%Rest eines alten Versuchts, die Dimension in (2) zu zeigen:
									%$B_\eps(x) \cap \aff(B)$ is a $j+l$ dimensional ball in the affine linear space $\aff(B)$. Since $S$ almost surely is in general position with respect to $\aff(B)$, the section
									%$$
									%S\cap B_\eps(x) \cap \aff(B)
									%$$
									%has dimension $j$.
									%Since the Ball $B_\eps(X)$ has full dimension and since it is centred in the affine linear space $S\cap \aff(B)$, the affine space $S\cap \aff(B)$ itself also has dimension j and so does $S\cap B$. %\dim(S\cap B)= \dim(S\cap \aff(B)) folgt aus S\cap \relint(B) \neq \emptyset?								
\end{proof}

\subsection{Probabilities that fixed faces get intersected - Proofs}
\begin{proof}[Proof of Proposition \ref{prop:Prob_S_cap_F_notempty_Cube}]
	The subspace $L$ intersects the face $B$ if and only if it intersects the cone $C$ spanned by $B$ not only at the origin.
	Hence, by the conic Crofton formula (see Theorem \ref{thm:conic_crofton}) we have
	\begin{align}
	\P(L\cap B \neq \emptyset )
	=
	\P(L \cap C \neq \{0\})
	=
	2(\upsilon_{l+1}(C)+\upsilon_{l+3}(C)+\ldots).\label{eq:random_sections_cones_first_step}
	\end{align}
	Since $F$ is a face of $[-1,1]^n$, it is a cube itself and $C$ is a convex cone isometric to $\Ccu_d(\sigma^2)$, where
	%by definition of $\Ccu_d(\sigma^2)$,
	the parameter $\sigma$ is the distance $\dist(\{0\}, B)$ between the origin and the cube $B$. It is easy to see that this distance is $\sqrt{n-d}$ and hence $C$ is isometric to $\Ccu_d(n-d)$, which completes the proof.
\end{proof}

\begin{proof}[Proof of Proposition \ref{prop:Prob_S_cap_F_notempty_crosspoly}]
As in the proof of Proposition \ref{prop:Prob_S_cap_F_notempty_Cube} let $C$ be the convex cone that is spanned by $B$. As before $L$ and $B$ are disjoint if and only if the intersection of $L$ and $C$ is only the origin. It follows that  \eqref{eq:random_sections_cones_first_step} holds.
	
	By the symmetry of the problem it is obvious that the probability $\P(L\cap B \neq \emptyset )$ depends on the face $B$ only by its dimension. Hence, without loss of generality we can assume that $B$ has the form $B=\conv \{\eee_1,\ldots, \eee_{d+1}\}$, which by definition is a regular simplex.
In this case  $C$ is isometric to $\Cs_{d+1}(0)$ and thus by Theorem~\ref{thm:conic_crofton}
	%Alles falsch:
	%Hence C is isometric to $\Cs_n(\sigma^2)$, where $\sigma$ is the distance between $F$ and the origin. Putting in $\dist(\{0\}, F)=\frac{1}{\sqrt{d-1}}$ we have $C \cong \Cs_n(\frac{1}{d-1})$.
	%This result can be formally proven by comparing the scalar products between the vectors spanning $C$ with those between the vectors spanning $\Cs_n(\frac{1}{d-1})$. Thus we have
	\begin{align*}
	\P(L\cap B \neq \emptyset )
	=
	2(\upsilon_{l+1}(C)+\upsilon_{l+3}(C)+\ldots)
	=
	%2 \left(
	%\upsilon_{l+1} \left(\Cs_d\left(\frac{1}{d+1}\right)\right) +\upsilon_{l+3}\left(\Cs_d\left(\frac{1}{d+1}\right)\right)+\ldots
	%\right).
	2 \left(
	\upsilon_{l+1}(\Cs_{d+1}(0))+\upsilon_{l+3}(\Cs_{d+1}(0))+\ldots
	\right) .
	\end{align*}
In~\cite[Propositions 1.3 and 1.4(d)]{kabluchko2017absorption}, it was shown that for $k\in\{0,\ldots,d+1\}$ the $k$-th intrinsic volume of $\Cs_{d+1}(0)$ is
	$$
	\upsilon_{k}(\Cs_{d+1}(0))= \binom{d+1}{k} 2^{-(d+1)},
	$$
	which gives
	\begin{align*}
	\P(L\cap B\neq \emptyset)
	=
	2^{-d} \left(
	\binom{d+1}{l+1} + \binom{d+1}{l+3} + \ldots
	\right)
	=
	2^{-d} \left(
	\binom{d}{l} + \binom{d}{l+1} + \ldots
	\right).
	\end{align*}
	
	%	Hence we have
	%	\begin{align*}
	%		\P(S \cap C \neq \{0\})
	%		=
	%		2^{-d} \left(
	%		\binom{d+1}{l-1} + \binom{d+1}{l-3} + \ldots
	%		\right)
	%	\end{align*}
\end{proof}

\begin{proof}[Proof of Proposition \ref{prop:Prob_S_cap_F_notempty_simplex}]
As is in the proofs of Propositions~\ref{prop:Prob_S_cap_F_notempty_Cube} and~\ref{prop:Prob_S_cap_F_notempty_crosspoly}, $L$ intersects the face $B$ if and only if it intersects the cone $C$ spanned by $B$ not only in the origin.
	Hence by the conic Crofton formula stated in Theorem \ref{thm:conic_crofton} we have \eqref{eq:random_sections_cones_first_step}.
	
Up to a factor, $P_n$ is isometric to the $n$-dimensional standard simplex
	$$S_{n} = \conv \{ \eee_1, \ldots, \eee_{n+1} \} \subset \R^{n+1}.$$
	Thus, denoting the centre of $S_{n}$ by $m=\frac{\eee_1+\ldots+\eee_{n+1}}{n+1}$, the cone $C$ is isometric to the cone $D$ spanned by the vectors $\eee_1-m,\ldots, \eee_{d+1}-m$.
	To determine the isometry type of $D$, and thus $C$, we just calculate the scalar products of the vectors spanning it. These are
	\begin{align*}
	\langle
	\eee_i -m, \eee_j -m
	\rangle
	=
	-\frac{1}{n+1} + \delta_{i, j} ,
	\end{align*}
	where we recall that $\delta_{i,j}=\indi{i=j}$ is the Kronecker delta.
	It follows that
	$
	D \cong \Cs_{d+1} \left( - \frac{1}{n+1} \right)
	$
	and thus
	$$
	C \cong \Cs_{d+1} \left( -\frac{1}{n+1} \right) .
	$$
	This completes the proof of the proposition.
\end{proof}

\subsection{Asymptotics}

\begin{proof}[Proof of Corollary \ref{cor:asymp_phis}]
	We start by analysing the term $\gKZ_{n-i+s+1}\left( - \frac{1}{n+l-s+1} \right)$, where
	$$
		\gKZ_k(r):= \P[\eta_1<0,\ldots,\eta_k<0],
	$$
	and $(\eta_1,\ldots,\eta_k)$ is a zero-mean Gaussian vector with
	$$
		\Cov(\eta_i,\eta_j)= r+\delta_{i,j}.
	$$
Note that $\gKZ_k(r)$ has the same meaning as in~\cite{kabluchko2017absorption} and $g_k(-r/(1+kr))= \gs_k(r)$.  	Recalling that by \cite[Proposition 1.2(b)]{kabluchko2017absorption} the internal solid angle $\beta(F,P_n)$ of an $(n-1)$-dimensional regular simplex $P_n$ at a $(k-1)$-dimensional face $F$ equals
	$$
		\beta(F,P_n)=\gKZ_{n-k}\left(-\frac{1}{n}\right),
	$$
	we can express our term by
	$$
		\gKZ_{n-i+s+1}\left( - \frac{1}{n-l+s+1} \right)
		=
		\gKZ_{n-i+s+1}\left(-\frac{1}{(n-i+s+1)+(i-l)}\right)
		=
		\beta (F,T),
	$$
	where $T:=\conv\{\eee_1,\ldots,\eee_{n-l+s+1}\}$ is an $(n-l+s)$-dimensional simplex and $F\in \cF_{i-l-1}(T)$ is one of its $(i-l-1)$-faces.
	In other words, we need the asymptotic behaviour of the internal solid angle of a simplex with increasing dimension at a face of fixed dimension. Such a formula has been derived in \cite[Corollary 2.1]{boroczky}. It states that for $n\to \infty$
	\begin{align}
	%\gKZ_{n-i-s+1}\left( - \frac{1}{n-l-s+1} \right)
	%=
		\beta(F,T)
		&=
		\frac{(i-l)^{\frac{n+s-i}{2}} \exp(\frac{n+s+2l-3i+1}{2})}{2^{\frac{n+s-i+2}{2}} \pi^{\frac{n+s-i+1}{2}} (n-l+s+1)^{\frac{n+s-i}{2}} } \left( 1+O\left( \frac{(i-l-1)^2+1}{n-l+s+1} \right) \right) \notag
		\\
		&\sim
		(2 \sqrt{\pi})^{-1} \eee^\frac{3l-3i}{2} \left(\frac{2\pi}{i-l}\right)^\frac{i-s}{2} \cdot n^{-\frac{n}{2}} n^\frac{i-s}{2} \left(\frac{(i-l)\eee}{2\pi}\right)^\frac{n}{2}. \label{eq:gs_asmyp}
	\end{align}
	In the last step the only non-trivial formula we used is
	$$
		(n-l+s+1)^{\frac{n+s-i}{2}}= (n-l+s+1)^{\frac{n}{2}} \cdot (n-l+s+1)^{\frac{s-i}{2}} \sim \eee^{\frac{s-l+1}{2}} n^\frac{n}{2} \cdot n^{\frac{s-i}{2}}.
	$$
	Having \eqref{eq:gs_asmyp} we can prove the corollary's actual statement.
	
	Recalling that by Theorem \ref{thm:random_sections_polytopes}
	$$
		\E\phis(n-i,n-l,n)=
		2\binom{n+1}{n-i+l+1}
		\sum_{s=1,3,5,\ldots} \upsilon_{l+s} \left(\Cs_{n-i+l+1} \left( -\frac{1}{n+1} \right)\right)
		$$
	and that the sum over all odd intrinsic volumes as well as the sum over all even intrinsic volumes equals $1/2$, we have
	\begin{align}
		\binom{n+1}{n-i+l+1} - \phis(n-i,n-l,n)= 2\binom{n+1}{n-i+l+1}
		\sum_{s=1,3,5,\ldots} \upsilon_{l-s} \left(\Cs_{n-i+l+1} \left( -\frac{1}{n+1} \right)\right). \label{eq:phis_asymp_def}
	\end{align}
	With the formula $\upsilon_k(\Cs_n(r))= \binom{n}{k} \gKZ_k \left(-\frac{r}{1+kr}\right)  \gKZ_{n-k} \left(\frac{r}{1+kr}\right)$ from \cite[Proposition 1.3]{kabluchko2017absorption}
	we can express the summands by
	\begin{align*}
		&2 \binom{n+1}{n-i+l+1} \upsilon_{l-s} \left(\Cs_{n-i+l+1} \left( -\frac{1}{n+1} \right)\right) \\
		&=
		2 \binom{n+1}{n-i+l+1} \binom{n-i+l+1}{l-s} \gKZ_{l-s} \left(\frac{1}{n+1-l+s} \right)	\gKZ_{n-i+s+1} \left(- \frac{1}{n+1-l+s} \right).
	\end{align*}
With $\gKZ_{l-s} \left( \frac{1}{n+1-l+s} \right) \ton \gKZ_{l-s}(0)=2^{s-l}$	by~\cite[Proposition 1.4(d)]{kabluchko2017absorption} and using the asymptotics \eqref{eq:gs_asmyp},
	$$
		\binom{n+1}{n-i+l+1} \sim \frac{n^{i-l}}{(i-l)!}
		\qquad \text{ and } \qquad
		\binom{n-i+l+1}{l-s} \sim \frac{n^{l-s}}{(l-s)!}
	$$
	this is asymptotically equivalent to
	$$
		\frac{2^{s-l}\eee^\frac{3l-3i}{2}}{(i-l)!(l-s)!\sqrt{\pi}}
		\left(\frac{2\pi}{i-l}\right)^\frac{i-s}{2} \cdot n^{-\frac{n}{2}} n^\frac{3i-3s}{2} \left(\frac{(i-l)\eee}{2\pi}\right)^\frac{n}{2}.
	$$
 In view of~\eqref{eq:phis_asymp_def} we are actually interested in the sum of these formulas over $s=1,3,5,\ldots$ and $s\leq l$. Note that the summand with $s=1$ is of a larger order than every other summand. Since the number of summands is finite and does not depend on $n$, the sum is dominated by its largest summand.
	Thus
	\begin{align*}
		\binom{n+1}{n-i+l+1} - \phis(n-i,n-l,n)
		&\sim
		2\binom{n+1}{n-i+l+1}
		\upsilon_{l-1} \left(\Cs_{n-i+l+1} \left( -\frac{1}{n+1} \right)\right)\\
		&\sim
		\frac{\pi^{\frac{i-2}{2}} 2^\frac{i-2l+1}{2} \eee^{\frac{3l-3i}{2}}}{(l-1)!(i-l)! (i-l)^\frac{i-1}{2}}
		\cdot
		n^{-\frac{n}{2}} n^{\frac{3i-3}{2}} \left( \frac{(i-l) \eee}{2\pi} \right)^{\frac{n}{2}}.
	\end{align*}
\end{proof}

\begin{proof}[Proof of Corollary \ref{cor:asymp_phicu_j_n-l_n}]
Recall from Theorem~\ref{thm:random_sections_polytopes} that
$$
\phicu(j,n-l,n) = 2^{n-l-j+1} \binom{n}{l+j} \sum_{s=1,3,5,\ldots} \upsilon_{l+s}\left( \Ccu_{l+j}(n-l-j) \right).
$$
The intrinsic volumes on the right-hand side are given by Theorem~\ref{thm:intrinsic_volumes} as follows:
$$
\upsilon_{l+s}\left( \Ccu_{l+j}(n-l-j) \right)
=
2^{j-s+1} \binom{l+j}{l+s-1} \gcu_{l+s-1}(n-l-s+1) \gs_{j-s+1}\left( \frac{1}{n-l-j} \right),
$$
where we assume that $s\leq j+1$ because otherwise the intrinsic volumes vanish.
	We start by analysing the asymptotic behavior of the term
	$$
		\gcu_{l+s-1}(n-l-s+1)=\P\left[ \xi_{l+s} \geq \sqrt{n-l-s+1} \max_{1\le i\le l+s-1} |\xi_{i}|\right],
	$$
	where $\xi_1,\xi_2,\ldots$ are i.i.d.\ standard normal distributed random variables. Expressed as an integral it has the form
	\begin{align*}
		\gcu_{l+s-1}(n-l-s+1)
		=
		\int_0^\infty \varphi(t) F^{l+s-1}\left(\frac{t}{\sqrt{n-l-s+1}}\right) \dd t, %\\
		%=
		%\sqrt{n-l+s+1} &\int_0^\infty \varphi(r \sqrt{n-l+s+1}) F^{l+s-1}(r) \dd r,
	\end{align*}
	where $\varphi$ is the density of the standard normal distribution and $F$ is the cumulative distribution function of $|\xi_1|$.
	
	Since as $n \to \infty$, while $t>0$ stays constant, we have
	\begin{align}
		F\left(\frac{t}{\sqrt{n-l-s+1}}\right)
		=
		\sqrt{\frac{2}{\pi}}\int_{0}^{\frac{t}{\sqrt{n-l-s+1}}} \eee^{-\frac{x^2}{2}} \dd x
		\sim
		\sqrt{\frac{2}{\pi}} \frac{t}{\sqrt{n-l-s+1}}
		\sim
		\sqrt{\frac{2}{\pi}} \frac{t}{\sqrt{n}},
	\end{align}
	the integrand satisfies
	$$
		\varphi(t) F^{l+s-1}\left(\frac{t}{\sqrt{n-l-s+1}}\right)
		\sim
		\frac{1}{n^\frac{l+s-1}{2}}      \varphi(t) \left(\sqrt{\frac{2}{\pi}} t \right)^{l+s-1}
	$$
for every fixed $t>0$.
	The natural hypothesis, i.e.\ that as $n\to \infty$
	\begin{align}
		\gcu_{l+s-1}(n-l-s+1)
	%	=
	%	\int_0^\infty \varphi(t) F^{l+s-1}\left(\frac{t}{\sqrt{n-l+s+1}}\right) \dd t
		\sim
		\int_0^\infty\frac{1}{n^\frac{l+s-1}{2}} \varphi(t) \left(\sqrt{\frac{2}{\pi}} t \right)^{l+s-1} \dd t \label{eq:gcu_asymp_int}
	\end{align}
	follows from the dominated convergence theorem.
%Indeed, since $\varphi \leq \sqrt{\frac{2}{\pi}}$, the estimate
%	$\varphi(t) \left(\sqrt{\frac{2}{\pi}} t \right)^{l+s-1} \leq \left(\sqrt{\frac{2}{\pi}}\right)^{l+s} t^{l+s-1}$ holds and thus by dominated convergence we have
%	$$
%		n^\frac{l+s-1}{2} \int_0^\infty \varphi(t) F^{l+s-1}\left(\frac{t}{\sqrt{n-l+s+1}}\right) \dd t
%		\ton
%		\int_0^\infty \varphi(t) \left(\sqrt{\frac{2}{\pi}} t \right)^{l+s-1} \dd t
%	$$
%	proving \eqref{eq:gcu_asymp_int}.
Expressing the $(l+s-1)$\textsuperscript{th} moment of $|\xi_1|$ as
	$$
		\sqrt{\frac{2}{\pi}} \int_{0}^{\infty} t^{l+s-1} e^{-\frac{t^2}{2}} \dd t
		=
		\frac{2^{\frac{l+s-1}{2}} \Gamma\left( \frac{l+s}{2} \right)}{\sqrt{\pi}},
	$$
	we can simplify \eqref{eq:gcu_asymp_int} to
\begin{align*}
\gcu_{l+s-1}(n-l-s+1)
%		&\sim
%		\frac{1}{n^\frac{l+s-1}{2}} \left(\sqrt{\frac{2}{\pi}} \right)^{l+s-2}  \sqrt{\frac{2}{\pi}} \int_0^\infty t^{l+s-1} \varphi(t) \dd t \\
%		&=
%		\frac{1}{n^\frac{l+s-1}{2}} \left(\sqrt{\frac{2}{\pi}} \right)^{l+s-2}  \frac{2^{\frac{l+s-1}{2}} \Gamma\left( \frac{l+s}{2} \right)}{\sqrt{\pi}}\\
%		&=
\sim
\frac{1}{n^\frac{l+s-1}{2}} \left(\frac{2}{\sqrt{\pi}} \right)^{l+s-1}   \frac{\Gamma\left( \frac{l+s}{2} \right)}{2 \sqrt\pi}.
	\end{align*}
	
Using the continuity of $\gs_{j-s+1}$ and \cite[Proposition 1.4(d)]{kabluchko2017absorption} we have $\lim_{n\to\infty} \gs_{j-s+1}(\frac{1}{n-l-j}) = 2^{s-j-1}$ and thus the intrinsic volumes satisfy
\begin{align*}
		\upsilon_{l+s}\left( \Ccu_{l+j}(n-l-j) \right)
		&=
		2^{j-s+1} \binom{l+j}{l+s-1} \gcu_{l+s-1}(n-l-s+1) \gs_{j-s+1}\left( \frac{1}{n-l-j} \right) \\
		&\sim
		\binom{l+j}{l+s-1} \left(\frac{2}{\sqrt{\pi}} \right)^{l+s-1}   \frac{\Gamma\left( \frac{l+s}{2} \right)}{2\sqrt{\pi}} \cdot \frac{1}{n^\frac{l+s-1}{2}}
	\end{align*}
	as $n\to \infty$. Recalling that
$$
\phicu(j,n-l,n) = 2^{n-l-j+1} \binom{n}{l+j} \sum_{s=1,3,5,\ldots} \upsilon_{l+s}\left( \Ccu_{l+j}(n-l-j) \right)
$$
we conclude that the summand with $s=1$ is of a higher order than any other one.  Observe also that the number of non-zero summands is bounded by a term not depending on $n$. Thus, the whole sum is dominated by the first  summand, i.e.\
\begin{align*}
\phicu(j,n-l,n)
&\sim
2^{n-l-j+1} \binom{n}{l+j}
\upsilon_{l+1}\left( \Ccu_{l+j}(n-l-j) \right)\\
&\sim
2^{n-j} \binom{n}{l+j} \binom{l+j}{l} \left(\frac{1}{\sqrt{\pi}} \right)^{l+1}   \Gamma\left( \frac{l+1}{2} \right) \cdot \frac{1}{n^\frac{l}{2}}\\
&\sim
2^{n-j} \frac{n^{j+(l/2)}}{l!j!}  \frac{\Gamma\left( \frac{l+1}{2} \right)}{\pi^{(l+1)/2}}
\end{align*}
as $n\to\infty$.
\end{proof}

\subsection*{Acknowledgement}
We thank Jonathan Thalmann for pointing out a mistake in a former version of Proposition \ref{prop:inner_outer_angle} and the unknown referee for a careful reading of the manuscript.

\bibliography{absorption_probablity_bib}
\bibliographystyle{plainnat}

\end{document}